\documentclass[12pt,leqno]{amsart}
\usepackage{amsmath,amssymb,amsfonts}
\usepackage{eucal,enumerate}
\usepackage[all]{xy}
\usepackage{graphicx}

\newtheorem{lem}{Lemma}[section]
\newtheorem{cor}[lem]{Corollary}
\newtheorem{prop}[lem]{Proposition}
\newtheorem{thm}[lem]{Theorem}

\theoremstyle{definition}
\newtheorem{defn}[lem]{Definition}
\newtheorem{exam}[lem]{Example}

\numberwithin{equation}{section}
\numberwithin{figure}{section}

\renewcommand{\phi}{\varphi}                 
\renewcommand{\epsilon}{\varepsilon}
\newcommand\eps{\varepsilon}
\newcommand\eset{\varnothing}
\newcommand\setm{\smallsetminus}

\newcommand\inv{^{-1}}
\newcommand\trans{^{\text{\rm T}}}
\newcommand\textb{\text{\rm b}}
\newcommand\pib{\pi_{\textb}}
\newcommand\chib{\chi^{\textb}}
\newcommand\Lat{\operatorname{Lat}}
\newcommand\Latb{\operatorname{\Lat^{\textb}}}
\newcommand\clos{\operatorname{clos}}
\newcommand\bcl{\operatorname{bcl}}
\newcommand\rk{\operatorname{rk}}
\newcommand\codim{\operatorname{codim}}
\newcommand\nul{\operatorname{nul}}
\newcommand \Diag{\operatorname{Diag}}

\newcommand\cA{\mathcal{A}}
\newcommand\cB{\mathcal{B}}

\newcommand\cD{\mathcal{D}}
\newcommand\cJ{\mathcal{J}}
\newcommand\cL{\mathcal{L}}
\newcommand\cS{\mathcal{S}}

\newcommand\bbA{\mathbb{A}}

\newcommand\bbP{\mathbb{P}}
\newcommand\bbR{\mathbb{R}}
\newcommand\bbZ{\mathbb{Z}}

\newcommand \fG{\mathfrak G}
\newcommand\Ups{\Upsilon}
\newcommand\G{\Gamma}
\newcommand\s{\sigma}
\renewcommand\S{\Sigma}
\newcommand\Ends{\operatorname{Ends}}
\newcommand\K{\mathbf{K}}
\newcommand\e{\mathbf{b}}
\newcommand\bfa{\mathbf{a}}

\newcommand\bfy{\mathbf{y}}
\newcommand\bfz{\mathbf{z}}
\newcommand\bfF{\mathbf{F}}
\newcommand\bfL{\mathbf{L}}
\newcommand\M{\mathbf{M}}

\newcommand\0{\mathbf{0}}
\newcommand\h{\mathbf{h}}
\newcommand\x{\mathbf{x}}
\newcommand\z{\mathbf{z}}

\newcommand\sectionpage{\newpage}

\begin{document}
\allowdisplaybreaks

\title[Gain Signed Graphs]{Matroids of Gain Signed Graphs}
\author[Anderson, Su, Zaslavsky]{Laura Anderson, Ting Su, and Thomas Zaslavsky}
\address{Department of Mathematics \\ Binghamton University \\ 
Binghamton, New York, U.S.A. 13902-6000}
\address{Present affiliation of Ting Su: Changjiang Geophysical Exploration and Testing Co., China}
\email{laura@math.binghamton.edu, tsu2@binghamton.edu, zaslav@math.binghamton.edu}

\begin{abstract}
A signed graph has edge signs.  A gain graph has oriented edge gains drawn from a group.  We define the combination of the two for the abelian case, in which each oriented edge of a signed graph has a gain from an abelian group, concentrating on the case of the additive group of a field.  
We develop the elementary graph properties, the associated matroid, and the vector and hyperplanar representations.
\end{abstract}



\date{\today}

\maketitle

\vspace{-.6cm}
\setcounter{tocdepth}{3}
\tableofcontents

\sectionpage
\section{Introduction}\label{sec:intro}

An \emph{arrangement of hyperplanes} is a finite set of hyperplanes in some linear, affine, or projective space.  Hyperplane arrangements formed a rather obscure corner of mathematics as recently as Gr\"unbaum's famous book of 1967 \cite{Grunbaum}, in which they were the subject of a chapter, but subsequently they have expanded into many parts of combinatorics, combinatorial geometry, algebraic geometry, algebra, and even analysis.  A fundamental combinatorial object associated to an arrangement is its intersection semilattice, which consists of all intersections of subarrangements (excluding empty sets in the affine case).  Valuable information about the arrangement can be extracted from this semilattice; in particular, through the characteristic polynomial of the semilattice one obtains the number of regions in the complement of a real arrangement \cite{FUTA} and the cohomology of the complement of a complex arrangement \cite{OT}.

Certain special arrangements were long well known, and have been a basis for this expansion: the hyperplane arrangements $A_{n-1}^*$, $B_n^* = C_n^*$, and $D_n^*$ dual to the classical root systems.  
Greene observed that subarrangements of the first of these correspond to graphs and their matroids.  The hyperplanes have equations of the form $x_j=x_i$, and such a hyperplane corresponds to an edge $v_iv_j$ in the associated graph \cite{Greene}.  We call such a hyperplane \emph{graphic}.  Zaslavsky extended this notion to subarrangements of $B_n^* = C_n^*$ and $D_n^*$, developing a matroid and coloring theory of signed graphs---graphs in which each edge is signed positive or negative \cite{SG}---that facilitates the calculation of the characteristic polynomial of the arrangement \cite{SGC}.  Here hyperplanes have equations of the forms $x_j=\pm x_i$ and $x_i=0$, corresponding to signed edges.  

As time passed new particular arrangements attracted interest.  The Catalan arrangement (arising from combinatorial geometry) has hyperplanes of the form $x_j-x_i=0,\pm1$ for $i<j$; the Shi arrangement (arising from algebraic geometry) has hyperplanes of the form $x_j-x_i = 0, 1$ and the Linial arrangement has hyperplanes $x_j-x_i = 1$, both also for $i<j$.  These examples led to the idea of a \emph{deformation} of a root system arrangement: the hyperplanes may be translated (in multiple ways); for instance, the Catalan, Shi, and Linial arrangements are deformations of the $A_{n-1}^*$ arrangement.  (We call arrangements of hyperplanes of the form $x_j-x_i=c$ \emph{affinographic} since they are affine deformations of graphic arrangements.)  The same idea extends to the other classical root system arrangements, with hyperplanes of the form $x_j\pm x_i=c$, and there are variations such as the threshold arrangement with hyperplanes $x_j + x_i = 1$.

To extract such valuable information as mentioned above one has to calculate the characteristic polynomial.  This can be hard.  Athanasiadis in a major paper \cite{Athan} did that for many examples by using what he called the ``finite field method''.  Given an integral deformation of a rational linear arrangement, one treats it as an arrangement over a finite field of sufficiently large order---though in practice only prime fields are used.  This retains the intersection semilattice and enables one to compute the characteristic polynomial by a counting process.  Since there are infinitely many large primes, with sufficient ingenuity one obtains the polynomial for infinitely many prime arguments, thus determining it.

Meanwhile, Zaslavsky had generalized signed graphs, coloring, their hyperplane representations by subarrangements of $B_n^*$, and their matroids, to gain graphs \cite{BG1, BG2, BG3}.  In a gain graph, the edges of the graph are labelled by the elements of a group, invertibly, which means the group element depends on direction, being inverted when the direction is reversed.  Signed graphs are the case where the group has order 2.  A gain graph has a natural matroid, in fact, two of them, called the frame and lift matroids.  Zaslavsky noticed that the integral affinographic arrangements studied by Athanasiadis---that is, arrangements with equations $x_j-x_i=c$ with $c \in \bbZ$---are hyperplanar representations of gain graphs with gain group $\bbZ^+$, the additive integers, and the lattice structure of the arrangement is given by the lift matroid of the gain graph \cite{BG4, ECR}.  Athanasiadis' finite field method, which in practice uses only prime fields, is really a modular coloring method, in which the gain group is the cyclic group $\bbZ_m$ for all sufficiently large moduli $m$; having a finite group enables the necessary counting process by way of coloring using color set $\bbZ_m$ \cite{ECR}.  This simplifies the theory for all integral affinographic arrangements as there is no need for $m$ to be prime; it also strengthens the conclusions since it is possible, at least in principle, to determine all valid moduli $m$ by knowing the gain graph.  (These improvements do not reduce the need for ingenuity.)  However, it leaves out Athanasiadis' examples with hyperplanes of the form $x_j+x_i=c$ in \cite[Theorems 3.10 \emph{et seq.}\ and 5.4--5.5]{Athan}, because those hyperplanes cannot be treated with gain graphs.  (We call such arrangements \emph{signed affinographic} because they are affine deformations of arrangements that represent signed graphs.)  The problem is that, while $x_j-x_i=c$ corresponds to an ordinary or positive edge with gain $c$, $x_j+x_i=c$ corresponds to a negative edge with gain $c$.  Neither signed nor gain graphs alone can handle this; it calls for a combination.  How to produce the necessary combination remained an unsolved problem for some decades.

That is the problem we solve here.  However, the immediate inspiration was different.  Ohsugi and Hibi were led by an algebraic question to study the edge polytope of a graph \cite{OH1}.  This is the convex hull of vectors $\e_i+\e_j$, one for each edge $v_iv_j$ in the graph.  Part of their work involved finding the affine span of a set of such vectors \cite[Proposition 1.3]{OH1}.  It was clear, from the viewpoint of signed graphs, that the affine span is a signed-graphic property; but what, exactly, is that property?  We quickly realized that it requires a combination of signs and identically-$1$ gains on the edges of the graph.  This gave us the necessary hint, and from that we produced a successful definition of a gain signed graph and (the proof of success) a succession of theorems under the assumption that we have an abelian gain group, such as the additive group of a field.  This paper reports our results.  As a demonstration we prove Ohsugi and Hibi's proposition in our new way in Example \ref{X:edgepolytope} and generalize it to bidirected graphs.

\emph{Terminological note.}  Gain signed graphs are not the same as signed gain graphs.  The inverse order of modifiers corresponds grammatically to the order of structures.  A signed gain graph is a gain graph with added signs; the gain-graph structure does not involve the signs, though the sign structure might depend on the gains.  We believe that trying to define signed gain graphs has previously been the barrier to a successful combination of signs and gains.  It is the realization that the signs influence the gain structure, so one needs gains superimposed on a signed graph, that enabled us to succeed.

\sectionpage
\section{Technical introduction}\label{sec:tech}

The object of study is a \emph{gain signed graph}, which is a triple $\Ups = (\G,\s, \phi)$ where $\G=(V,E)$ is a graph, the \emph{signature} or sign function $\s$ gives each edge an element of the sign set $\{+1,-1\}$, and the \emph{gain function} $\phi$ is an oriented labelling of edges from a group; that is, inverting the edge inverts the gain.  The group will be the additive group of a field $\K$ until at the end we retroactively generalize in Section \ref{sec:abstract}.  All these terms must be explained further.

We assume acquaintance with the basics of matroid theory, such as in the beginning of Oxley \cite{Oxley}.  As there are inconsistencies amongst common usage in graph theory, usage in matroid theory, and our special needs, we summarize some of our vocabulary here.  The definitions will follow in this section.

\begin{enumerate}[\qquad ]
\item \emph{Graph}: it may have links, loops, and half and loose edges, as well as multiple edges.  We write $n:=|V|$.
\item \emph{Circle}: a connected, 2-regular subgraph or its edge set.
\item \emph{Circuit}: a matroid circuit. 
\item \emph{Sign circuit}: a circuit of the frame matroid $\bfF(\G,\s)$.  Similarly for cocircuits.
\item \emph{Sign cycle}: a sign circuit oriented to have no source or sink.
\item $\bfF(\S)$: the frame matroid of a signed graph $\S$.
\item A sign circuit is \emph{neutral} in $\Ups$ if it has gain 0.
\item $\M(\Ups)$: the matroid of the gain signed graph $\Ups$.
\item $\M_\infty(\Ups)$: the extended matroid of $\Ups$, with extra element $e_\infty$.
\item \emph{Hypercircuit, hypercocircuit}: a circuit or cocircuit of the matroid $\M(\Upsilon)$ or $\M_\infty(\Ups)$, as appropriate.
\item $\Lat M$: the geometric lattice of a matroid $M$.
\end{enumerate}

\subsection{Graphs}\label{sec:graph}\

The graph $\G$ always means $(V,E)$ whose vertex set is $V=V(\G)=\{v_1,v_2,\ldots,v_n\}$, with $n:=|V|$.  
Our notion of a graph is unusual because we allow four kinds of edge: links, loops, half edges, and loose edges.  Multiple edges are allowed.  A link or loop has two \emph{ends}, each of which is incident with one vertex, called an endpoint of the edge.  (Note that an endpoint is a vertex while an end is part of an edge.)  The edge is a link if these two end vertices are distinct and a loop if they coincide.  A half edge has one end, which is incident with one vertex, and a loose edge has no ends.  In notation we often write $e_{vw}$ or $e_{ij}$ to indicate the endpoints of a link or loop $e$, $e_v$ or $e_i$ for a half edge $e$, and $e_\eset$ for a loose edge (although this notation does not distinguish parallel edges, which have the same endpoints).  
We write $E^2$ for the set of links and loops, $E^0$ for that of loose edges, and $E^1$ for that of half edges.  
For an end of edge $e$ incident with a vertex $v$ we write $(v,e)$; this notation is technically ambiguous for a loop but that will rarely cause difficulty.  We write $\Ends(\G)$ for the set of edge ends of $\G$.  

A \emph{simple graph} is a graph that has only links and has no parallel edges.

The \emph{components} of $\G$ are the usual connected components, not including loose edges.  More precisely, they are the \emph{vertex components} of $\G$.  The \emph{edge components} are each loose edge and the vertex components that are not isolated vertices.

A \emph{circle} is a connected 2-regular graph, or its edge set (widely known as a ``cycle'' or ``circuit'', but we have different uses for those names).  A \emph{unicycle} is a circle or half edge that may have trees attached.  A \emph{handcuff} is a pair of circles with at most one common vertex together with a minimal connecting path (of length 0 if there is a common vertex).

A \emph{cut} is a non-empty set that consists of all edges between a subset of $V$ and its complement.  A minimal cut is a \emph{bond}.  The two components of the bond-deleted graph that are joined by the bond are its \emph{sides}.  

The edge set induced by $X \subseteq V$ is denoted by $E{:}X$; the induced subgraph is denoted by $\G{:}X$.  Suppose $\pi$ is a partition or partial partition of $V$ (that is, a partition of a subset); then the partition-induced subgraph is denoted by $\G{:}\pi = (\bigcup\pi, E{:}\pi) := \bigcup_{B\in\pi} \G{:}B$.  

The \emph{restriction} of $\G$ to an edge set $S$ is the spanning subgraph $\G|S := (V,S)$.  
The partition of $V$ induced by the components of $\G$ is $\pi(\G)$; for an edge set $S \subseteq E$ the induced partition of $V$ is $\pi(S) := \pi(\G|S) = \pi(V,S)$.  The number of components of $\G$ is $c(\G) = |\pi(\G)|$.  
The set of vertices of edges of $S$ is $V(S)$, and $c(S)$ is the number of components of the spanning subgraph $\G|S$.  The \emph{cyclomatic number} of $S$ is 
$$
\xi(S) = |S| - |V| + c(S) = |S| - |V(\G|S)| + c(\G|S).
$$
It is the number of edges that must be deleted in order to eliminate all circuits of the graphic matroid; this includes loose edges.

For a walk, say $W=u_0e_1u_1e_2\cdots e_lu_l$ (where the length $l\geq0$), we write $W_{ij}$ to denote that part of $W$ beginning at $u_i$ and ending at $u_j$.  The \emph{reverse} of that partial walk is $W_{ji}=W_{ij}\inv$.  In particular, $W=W_{0l}$ and $W\inv = W_{l0}$.  The direction of a walk $W$ is indicated by the ordered pair of initial and final vertices; note that this concept of direction is not the one common in directed graph theory and is different from the concept of orientation in Section \ref{sec:orient}.  

For sign circuits (to be defined shortly) we sometimes need an extension of the notion of walk.  An \emph{ultrawalk} is like a walk except that it may begin or end with a half edge.  For example, $W=e_0u_0e_1u_1e_2\cdots e_lu_l$ is an ultrawalk if $e_0$ is a half edge incident with $u_0$.

\subsection{Signed graphs}\label{sec:sg}\

A \emph{signed graph} $\S=(\G,\s)$ consists of a graph $\G$ and a function $\s : E \to \{+1,-1\}$, the \emph{signature} or \emph{sign function}, that is defined on all edges (this differs from the convention in previous articles such as \cite{SG}).  Links and loops may be positive or negative.  Half edges must be negative and loose edges must be positive.  The set of edges with sign $\eps$ is $E^\eps$.
A simple graph $\G$ gives rise to a signed graph in several ways: we may give all edges the same sign $\eps$ and we write that signed graph $\eps\G$, or we may double the edges with both signs, which is the graph $\pm\G = (+\G) \cup (-\G)$, which is the union of two edge-disjoint signed graphs on the same vertex set $V=V(\G)$.

We apply most graph notations to signed graphs; e.g., the restriction of a signed graph is $\S|S := (\G|S, \s|_S)$, more simply $(\G|S,\s)$.  

The sign of an edge set $S$ is $\s(S) :=$ the product of the signs of edges in $S$.  This is important for a circle $C$, which is either positive or negative.  If an edge set $S \subseteq E$ contains no negative circles and no half edges, it is \emph{sign balanced}, and similarly for a subgraph.  We write $b_\S(S)$ for the number of connected components of $S$, or more precisely of $\S|S$, that are sign balanced.

A signed graph $\S$ is \emph{sign antibalanced} if its negative, $-\S$, is balanced.

The sign of a walk, say $W=u_0e_1u_1e_2\cdots e_lu_l$, is $\s(W) := \prod_{i=1}^l \s(e_i)$.  Note that the sign of a walk is not necessarily the same as the sign of its edge set since the walk may repeat edges; but the two definitions do agree for a circle.  
The same definition gives the sign of an ultrawalk: it is the product of the signs of its edges, accounting for multiplicity (recall that half edges are negative).

The \emph{frame matroid} $\bfF(\S)$ (\cite{SG, BG2}; called the bias matroid in \cite{BG2}) is a matroid whose ground set is $E$ and whose circuits are the sign circuits of $\S$.  To define a sign circuit, first we define a \emph{negative figure}: it is any negative circle or half edge.  A \emph{sign circuit} is an edge set that is either a positive circle, or a loose edge, or a pair of negative figures that have exactly one common vertex (a \emph{tight handcuff}), or a pair of disjoint negative figures together with a minimal connecting path (a \emph{loose handcuff})---minimal in the sense that it intersects the two circles only at its endpoints; the length is immaterial.  
(We view a tight handcuff as having a connecting path of length $0$.)  
In this definition any negative circle can be replaced by a half edge; thus we call a \emph{negative figure} any negative circle or half edge.  
The frame matroid's rank function is $\rk_\S(S) = n - b_\S(S)$.  

\begin{figure}[htbp]
\includegraphics[scale=.7]{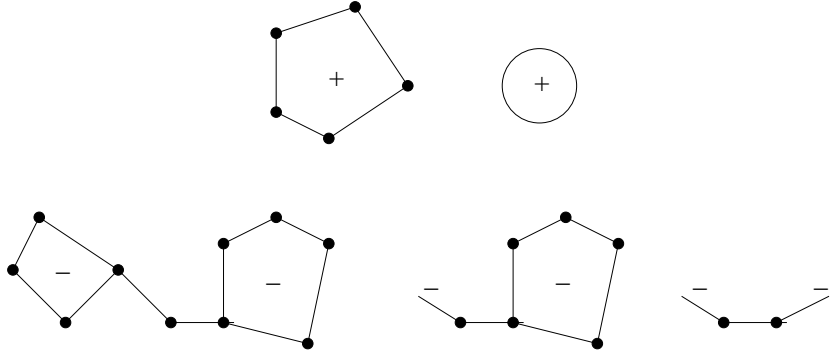}
\caption{The two different varieties of sign circuit.  In the first row, sign balanced, showing a positive circle and a (positive) loose edge.  In the second row, sign unbalanced, showing a handcuff with two negative circles, one such circle and one (negative) half edge, and two half edges; the connecting path may have length $0$.}
\label{F:signcircuit}
\end{figure}

A fundamental operation on signed graphs is \emph{switching} (which we call \emph{sign switching} in the context of gain signed graphs).  A \emph{switching function}, briefly \emph{switcher}, is any vertex function $\zeta: V \to \{\pm1\}$.  Switching $\S$ by $\zeta$ means changing the signs of links and loops: $\S = (\G,\s)$ becomes $\S^\zeta = (\G,\s^\zeta)$ whose sign function is $\s^\zeta(e_{vw}) := \zeta(v)\s(e)\zeta(w)$.  The signs of loose and half edges do not change.  Switching does not change the signs of closed walks and, most importantly, of circles.
A signed graph is balanced or antibalanced if and only if it can be sign switched to all positive or all negative \cite{SG}.

A switcher can also be regarded as a way of signing edges.  We call $\zeta$ a \emph{sign potential} for an edge set $S$ if $\s(e_{vw}) = \zeta(v)\inv\zeta(w)$ for every edge in $S$.  (The inversion exists to suit the common notion of a potential, although for signs it is unnecessary.)  For a subset $S \subseteq E^2$ this is equivalent to $\s^\zeta|_S$ being identically $+1$.

\begin{lem}[\cite{SG}]\label{L:signswitching}
A signed graph is balanced if and only if it switches to all positive.
\end{lem}

Note that, by our sign conventions, being all positive implies having no half edges.

There are two basic ways to make a signed graph $\S$ balanced.  One can delete some edge set $D$ whose complement is balanced---then $D$ is called a \emph{deletion set} for $\S$; or one can negate the signs of some edge set $N$ to obtain balance---then $N$ is called a \emph{negation set}.  Every negation set is obviously a deletion set, but not every deletion set is a negation set.  However, the minimal deletion and negation sets are the same.  A simple property of a minimal deletion (or negation) set is this:

\begin{lem}\label{L:sbalset}
Let $\S$ be a connected signed graph.  
The minimal deletion sets in $\S$ are the complements of the maximal balanced edge sets.
The complement of a minimal deletion set in $\S$ is connected.  
\end{lem}

For $S \subseteq E$ define the \emph{sign balance-closure}:%
\footnote{Not ``balanced closure''; it might be neither balanced nor a closure.}
$$
\bcl_\S(S) := S \cup \{e \notin S : \exists\ \text{positive circle } C \text{ such that } e \in C \subseteq S \cup \{e\} \} \cup E^0.
$$
We also define the partial partition of $V$ due to $S$, 
$$\pib(S) = \{ V(B):  B \text{ is a sign-balanced component of } \S|S \},$$ 
whose parts are the vertex sets of sign-balanced components of $\S|S$; 
$$U_\S(S) := V \setm \bigcup\pib(S),$$ 
the set of vertices of sign-unbalanced components of $S$;
and for a signature $\zeta: \bigcup\pib(S) \to \{\pm1\}$, the set 
$$E(\zeta) := \{ e_{vw} \in E^2{:}U_\S(S)^c: \s(e_{vw}) = \zeta(v)\zeta(w) \},$$
i.e., the edges for which $\zeta$ is a sign potential.  
With these notions we can define the frame matroid $\bfF(\S)$.

\begin{thm}[Frame Matroid of a Signed Graph {\cite[Theorem 5.1]{SG}}]\label{L:sgmatroid}
In the frame matroid of a signed graph $\S$, consider an edge set $S$.  

The rank function is
$$
\rk_\S(S) = n - b_\S(S) = |V(S)| - b_\S(V(S),S).
$$

The sign balance-closure of a balanced edge set is given by $\bcl_\S(S) = E(\zeta){:}\pib(S) \cup E^0$, where $\zeta$ switches $S$ to all positive.
The closure of any edge set is given by
$$
\clos_\S(S) = [E{:}U_\S(S)] \cup \bcl_\S(S{:}U_\S(S)^c) = [E{:}U_\S(S)] \cup [E(\zeta){:}\pib(S)] \cup E^0.
$$

The closed sets are those of the form 
$$
[E{:}U] \cup [E(\zeta){:}\pi] \cup E^0,
$$
where $U \subseteq V$, $\pi$ partitions $U^c$, and $\zeta: U^c \to \{\pm1\}$.

The cocircuits $D$ are any of the following four types:
\begin{enumerate}[{\rm \quad(D1)}]
\item\label{bondb} $D$ is a bond in a balanced component of $\S$.
\item\label{del} $D$ is a minimal deletion set of an unbalanced component of $\S$.
\item\label{bondu} $D$ is a cut in an unbalanced component of $\S$, such that one side is connected and balanced and the other has no balanced component.
\item\label{bondu-del} $D$ consists of a cut in an unbalanced component of $\S$, such that both sides are unbalanced, one side is connected and unbalanced, and the other side has no balanced component, together with a minimal deletion set of the first side.
\end{enumerate}
\end{thm}

Note that (D\ref{bondu-del}) incorporates (D\ref{bondu}) if we allow the deletion set in (D\ref{bondu-del}) to be empty.  However, for practical use it seems better to state (D\ref{bondu}) separately.

A vector representation of a signed graph is $\x: E(\S) \to \K^n$ defined by 
\begin{align*}
\x(e_{ij}) &:= \e_j - \s(e)\e_i,\\
\x(e_i) &:= \e_i, \\
\x(e_\eset) &:= \0.
\end{align*}
(A negative loop $e_{ii}$ is represented by $\pm2\e_i$.  Recall that we assume the field $\K$ has characteristic other than 2.)

\begin{thm}[Representation of Signed Graphs {\cite[Section 8B]{SG}}]\label{T:sgrep}
The mapping $\x$ is a vector representation of the frame matroid $\bfF(\S)$.
\end{thm}

\subsection{Orientation}\label{sec:orient}\

A signed graph, that is $\S = (\G,\s)$, is oriented differently from an ordinary, unsigned graph.  An \emph{orientation} of $\S$ is a mapping from edge ends to signs, say $\tau: \Ends \to \{+1,-1\}$, such that for a link or loop $e_{vw}$, $\tau(v,e)\tau(w,e) = -\s(e_{vw})$.  In diagrams we interpret an end sign $+1$ as an arrow pointing into the endpoint and sign $-1$ as an arrow pointing from the endpoint into the edge.  With this sign rule an oriented positive edge is an ordinary directed edge.  An oriented negative edge, on the contrary, does not have a direction; it is either \emph{introverted}, with both ends negative, or \emph{extraverted}, with both ends positive.  A half edge also has an orientation (at its single end); it is introverted if its end is negative and extraverted if it is positive.  
A loose edge has no orientation (as it has no end).  To denote an oriented signed graph and its orientation we write $\S_\tau$.  When we do not wish to specify a particular orientation of an edge we write $e$ for one orientation and $e\inv$ for the opposite orientation.

\begin{figure}[h]
\includegraphics{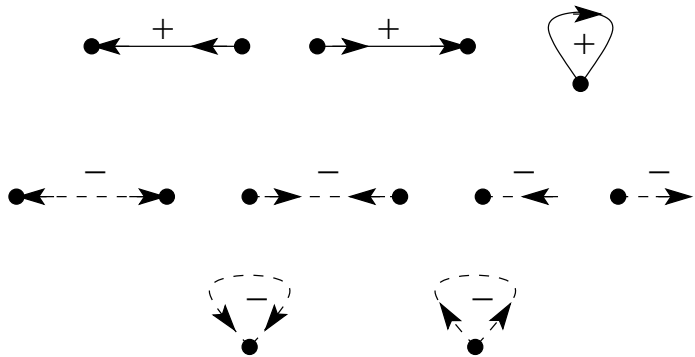}
\bigskip
\caption{Oriented edges.  Top: positive edges with $\tau = +-,\ -+,\ -+$.  Middle:  negative edges with $\tau = ++,\ --,\ +,\ -$. Bottom:  negative loops with $\tau = ++,\ --$.}
\label{F:oriented edges}
\end{figure}

Orientations switch: if we switch $\S$ by $\zeta$, an orientation $\tau$ becomes $\tau^\zeta$ defined by $\tau^\zeta(v,e) = \zeta(v)\tau(v,e)$.  Switching of an orientation applies to all edges (and is trivial for loose edges).

A walk in an oriented signed graph, say $W=u_0e_1u_1e_2\cdots e_lu_l$, is \emph{coherent} if the arrows line up at each vertex, that is, $\tau(u_i,e_i) = -\tau(u_i,e_{i+1})$ for each $i=1,2,\ldots,l-1$.  (In any walk, each intermediate vertex $u_i$ is either coherent, if $\tau(u_i,e_i) = -\tau(u_i,e_{i+1})$, or incoherent.  Coherence of a vertex in a walk is not altered by switching.)  The same definition applies to any ultrawalk if among the intermediate vertices we include any vertex supporting a half edge.  
A \emph{closed walk}, in which $u_l=u_0$, is \emph{coherent} if it is coherently oriented as an open walk and in addition it closes coherently, i.e., $\tau(u_l,e_l) = -\tau(u_0,e_1)$; this implies that $\s(W)=+1$.  A closed walk is \emph{rooted coherent} if it is coherent as an open walk but not as a closed walk; this implies that $\s(W) = -1$; the \emph{root} is $u_0$.  To prove those sign statements, we show that a positive closed walk, however oriented, has an even number of incoherent vertices, while a negative circle has an odd number.  The proof is a simple calculation  (with subscripts modulo $l$):
\begin{align*}
\s(C) &= \prod_{i=1}^l \s(e_i) = \prod_{i=1}^l [-\tau(u_{i-1},e_i)\tau(u_i,e_i)] 
\\&= \prod_{i=1}^l [-\tau(u_i,e_i)\tau(u_i,e_{i+1})] = \prod_{i=1}^l \iota(u_i),
\end{align*}
where $\iota(u_i) = +1$ or $-1$ if the vertex is, respectively, coherent or incoherent.

A \emph{source} in an oriented signed graph is a vertex $v$ such that all arrows point inward, i.e., $\tau(v,e) = +1$ for every edge end $(v,e)$ at $v$.  A \emph{sink} is similar but all arrows point outward, i.e., $\tau(v,e) = -1$ for every edge end $(v,e)$ at $v$.  
There are exactly two ways to orient a sign circuit (other than a loose edge) so that none of its vertices is a source or sink, and they differ only by reversing the orientations of all the edges.  We call such an edge set in an oriented signed graph a \emph{sign cycle}.

A sign circuit $C$ has one or two minimal covering ultrawalks, up to choice of direction and starting point, which we call \emph{circuit walks}.  The description can be complicated.  For a circle the circuit walk goes once around the circle.  For a tight handcuff whose negative figures are circles, it goes once around each circle, crossing over at their intersection vertex.  For a loose handcuff whose negative figures are circles, it goes once around each circle and twice through the connecting path, once in each direction; this reduces to the circuit walk of a tight handcuff if the connecting path has length zero.  For a handcuff in which one negative figure is a circle and the other is a half edge $h$, the circuit walk begins at $h$, goes through the connecting path (if any) to the circle, goes once around the circle, and returns along the connecting path to end with $h$.  In a handcuff whose negative figures are both half edges, the circuit walk is an ultrawalk that begins with one half edge and proceeds to the other, where it ends.  (The difference of the last type from the others is due to the fact that negative circles interfere with unimodularity while half edges do not.  This oddity does not affect our theory.  For a treatment of it see \cite{BZ}.)

For all types, if $C$ is oriented, it is a sign cycle if and only if one (equivalently, each) of these walks is coherent.

Now consider a circle.  If it is positive, it can be coherently oriented, and then it is a sign cycle.  If it is negative, it can be coherently oriented except at one arbitrarily chosen incoherent vertex $v$ with incident edges $e, f$.  Then we call it an \emph{excycle} \emph{rooted at $v$} if $\tau(v,e)=\tau(v,f)=+1$ (i.e., the edge ends at $v$ point towards $v$ and thus out of the circle) and an \emph{incycle} if the opposite.  An incycle or excycle is the nearest a negative circle can be to a cycle.  We also call an oriented half edge an incycle or excycle depending whether it is oriented into or out from its vertex.  
The orientation of a sign cycle on a loose handcuff is such that one circle is an incycle, the other is an excycle, and the connecting path is a coherent path connecting the root of the excycle to the root of the incycle and, at its endpoints, coherent with the incident circle edges.  For a tight handcuff sign cycle, the two roots are the common vertex and the incident edges in different circles are coherent with each other.  

\begin{figure}[htbp]
\includegraphics[scale=.7]{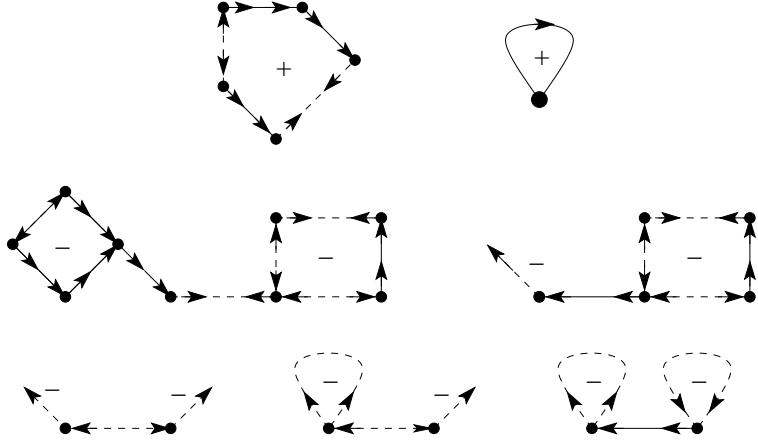}
\smallskip
\caption{Some sign cycles. Solid edges are positive and dashed edges are negative.  Top row:  a positive circle of length 5 and a positive loop.  Bottom rows:  sign-unbalanced sign circuits with different kinds of negative figures.}
\label{F:signcycles}
\end{figure}

\subsection{Gains}\label{sec:gains}\

\subsubsection{The gain function}\

The gain function $\phi$ is a mapping from oriented edges to an abelian group, written additively, such that the two orientations of the edge, $e$ and $e\inv$, have mutually inverse gains, i.e.,  $\phi(e\inv) = -\phi(e)$ (so when we specify a value for the gain $\phi(e)$, we must also specify the orientation).  
For a positive edge or a half edge, this means that the two directions of the edge have mutually inverse gains; but for a negative edge there is not a direction in the usual sense.  
An edge with gain 0 is called \emph{neutral}.  
The set of neutral loose edges is written $E^{00}$.  
Note that all edges have gains (unlike the definition in \cite{BG2}).  

In a slightly more formal definition, let the set of oriented edges be $\vec E$.  The gain function is a mapping $\phi$ from $\vec E$ to an additive abelian group with the property that $\phi(e\inv) = -\phi(e)$.  (It is important to remember that $e$ and $e\inv$ are the same object with opposite orientations.)

\emph{Reorienting} an edge in $\Ups$ means negating the $\tau$ values of that edge and inverting the gain of that edge as well.  We define sign-switching to have no effect on gains.

The restriction of a gain signed graph to an edge set $S$ is $\Ups|S := (\G|S, \s|_S, \phi|_S)$, abbreviated as $(\G|S,\s,\phi)$.  

Suppose we treat an unsigned graph as all positive; the pair $(\G,\phi)$ is called a \emph{gain graph}.  The frame matroid theory of signed graphs was developed in \cite{SG}; that of gain graphs was developed in \cite{BG2}.  The former is a special case of the latter; but it turns out that this is the wrong way to look at it, at least for us.  A gain signed graph is not a gain graph: the negative edges cannot be treated as edges of a gain graph.  The reason for our treatment of negative edges will appear when we introduce the vector model in Section \ref{sec:vector}.

Our gain group will be the additive group $\K^+$ of a field $\K$ whose characteristic is not equal to 2 (because we need $1 \neq -1$), especially the real or complex numbers.  We generalize this to an arbitrary abelian group in Section \ref{sec:abstract}.  The most interesting such groups may be the additive group of integers and those of the integers modulo a natural number.

\subsubsection{Walk gains}\label{sec:walkgains}\

We now define a crucial concept: the gain sum of a sign circuit $C$.  First we define that of a walk or ultrawalk $W$.  We want a definition that makes the gain of $W$ invariant under reorientation, that agrees with the gain of a walk in a gain graph, and that agrees with the vector model in Section \ref{sec:vector}; this implies a complicated definition that is not invariant under sign switching.  We choose an orientation $\tau$ of $\S$.  Then the gain of a walk $W=u_0e_1u_1e_2\cdots e_lu_l$ is 
\begin{equation}
\begin{aligned}
\phi(W) &:= - \tau(u_0,e_1) \sum_{i=1}^l \phi(e_i) \prod_{j=1}^{i-1} [-\tau(u_j,e_j)\tau(u_j,e_{j+1})] 
\\&= - \sum_{i=1}^l \phi(e_i) \s(W_{0,i-1}) \tau(u_{i-1},e_i),
\end{aligned}
\label{E:walkgainsum}
\end{equation}
where $W_{0,i-1}$ denotes the part of $W$ from $u_0$ to $u_{i-1}$.  The factor $[- \tau(u_0,e_1)]$ is a general sign correction that is positive if the first edge leaves $u_0$, negative if it enters $u_0$.  The factors $[-\tau(u_j,e_j)\tau(u_j,e_{j+1})]$ are positive at a coherent vertex $u_j$ and negative, thus reversing the signs with which following edge gains are added, at an incoherent vertex.  That is, $\phi(e_i)$ is added to the sum if it is preceded by an even number of incoherent vertices and subtracted if preceded by an odd number of incoherent vertices.

The gain of an ultrawalk is similar but we have to take account of some extreme cases, such as an ultrawalk that consists only of one half edge.  Let $W=e_0u_0e_1u_1e_2\cdots e_lu_le_{l+1}$ be an ultrawalk, in which $e_0$ and $e_{l+1}$ are half edges.  Its gain is given by 
\begin{equation}
\begin{aligned}
\phi(W) 
&:= \tau(u_0,e_0)\phi(e_0) + \phi(W_{0l}) - \phi(e_{l+1}) \s(W_{0l}) \tau(u_l,e_{l+1}),
\end{aligned}
\label{E:ultrawalkgainsum}
\end{equation}
where the first term is omitted if $W$ has no initial half edge and the last term is omitted if $W$ has no final half edge.

The gain of a walk of length $1$ may not equal the gain of the edge $e_1$; that depends on the orientation of $e_1$.  Considering the walk $u_0e_1u_1$, we see that 
\begin{equation}
\phi(u_0e_1u_1) = -\tau(u_0,e_1) \phi(e_1).
\label{E:walkgainsum1}
\end{equation}
Similarly, the gain of an ultrawalk $e_0u_0$ or $u_0e_1$ is given respectively by
\begin{equation}
\phi(e_0u_0) = \tau(u_0,e_0) \phi(e_0) \text{ and } \phi(u_0e_1) = -\tau(u_0,e_1) \phi(e_1).
\label{E:walkgainsum2}
\end{equation}

\begin{lem}\label{L:swwalkgain}
The gain of a walk or ultrawalk is invariant under reorientation and sign switching, except that it is negated by sign switching the initial vertex.
\end{lem}

\begin{proof}
Reorienting an edge $e_i$ changes the sign of both $\phi(e_i)$ and $\tau(u_{i-1},e_i)$, so it has no effect on the term of $e_i$.  For $j>i$ the sign of $W_{0,j-1}$ remains the same.  Thus, there is no change to $\phi(W)$.

Switching the initial vertex negates $\tau(u_0,e_1)$ (and $\tau(u_0,e_0)$ if there is an initial half edge or if $W$ is closed so $u_l=u_0$ and we can take $e_0=e_l$), so it negates the sum in \eqref{E:walkgainsum}.  Switching the final vertex has no effect on the sum.  Switching an internal vertex has no effect on coherence, so it does not change the sum.
\end{proof}

\begin{lem}\label{L:revwalkgain}
For a walk or ultrawalk $W$, the gain of\/ $W\inv$ is given by $\phi(W\inv) = -\s(W)\phi(W)$.
\end{lem}

\begin{proof}
We compute for a walk:
\begin{align*}
\phi(W) &=  - \sum_{i=1}^l \phi(e_i) \s(W_{0,i-1}) \tau(u_{i-1},e_i) 
\\&= - \sum_{i=1}^l \phi(e_i) \s(W)\s(W\inv_{l,i}) \s(e_i) \tau(u_{i-1},e_i) 
\\&= \s(W) \sum_{i=1}^l \phi(e_i) \s(W\inv_{l,i}) [\tau(u_i,e_i)] 
= -\s(W) \phi(W\inv).
\end{align*}
For an ultrawalk with half edges at both ends,
\begin{align*}
\phi(W) &=  \tau(u_0,e_0)\phi(e_0) + \phi(W_{0l}) - \phi(e_{l+1}) \s(W_{0l}) \tau(u_l,i_{l+1})
\\&= \s(W_{0l}) [ \s(W_{0l}) \tau(u_0,e_0)\phi(e_0) + \s(W_{0l}) \phi(W_{0l}) - \phi(e_{l+1}) \tau(u_l,i_{l+1}) ]
\\&= \s(W_{0l}) [ \phi(e_{l+1}) (-\tau(u_l,i_{l+1}) + \phi(W_{l0}) - \s(W_{0l}) (-\tau(u_0,e_0))\phi(e_0) ]
\\&= -\s(W) \phi(W\inv)
\end{align*}
by \eqref{E:ultrawalkgainsum}, because for the formula in the direction of $W\inv$, the half-edge orientations reverse.  For an ultrawalk with a half edge at only one end, omit one term in the preceding expressions.
\end{proof}

\begin{lem}\label{L:sumwalkgains}
If $W$ is the concatenation of walks, $W = W_1W_2$, then $\phi(W) = \phi(W_1) + \s(W_1)\phi(W_2)$.  The same holds for the concatenation of any two ultrawalks $W_1$ and $W_2$ for which $W_1$ ends and $W_2$ begins at the same vertex.
\end{lem}

\begin{proof}
This follows directly from the second form of Equation \ref{E:walkgainsum}, with appropriate modifications for half edges.
\end{proof}

The complicated definitions \eqref{E:walkgainsum} and \eqref{E:ultrawalkgainsum} call for explanation.  For simplicity assume $W$ is a path and $\tau(u_0,e_0)=+1$ ($e_0$ leaves $u_0$).  Consider first a path of positive edges, all oriented in the direction from the beginning at $u_0$ to the end at $u_l$.  Then $-\tau(u_{i-1},e_i) = +1 = \s(W_{0,i-1})$, so $\phi(W) = \sum_{i=1}^l \phi(e_i)$, simply a sum, consistent with an ordinary gain graph.  This consistency is what we want.  Suppose, though, that some of the edges (still all positive) are oriented in reverse; then at a backward edge we have $-\tau(u_{i-1},e_i) = -1$ so the gain of $e_i$ is subtracted, not added.  Equivalently, we can view this as first reorienting the backward edges, which negates their gains, and then summing the revised gains as in the first case.  Note that reorienting (say) the first backward edge $e_i$ changes $u_{i-1}$ from incoherent to coherent and reverses the coherence of $u_i$ (this is for $i>1$).  Thus, the number of incoherent vertices before $e_i$, and only before $e_i$, changes parity, while the gain changes sign.  This implies that the sign we put on $\phi(e_i)$ is determined by the number $p$ of preceding incoherent vertices in $W$ and the sum should contain $(-1)^p\phi(e_i)$, exactly as stated after Equation \eqref{E:walkgainsum}.

But what if our path $W$ contains negative edges?  We may switch it with a suitable sign-switching function $\zeta$ to $W^\zeta$ which is all positive; this does not change the gains of edges nor the state of coherence or incoherence at any vertex.  We may choose $\zeta$ not to switch the initial vertex, i.e., so that $\zeta(u_0)=+1$.  By Lemma \ref{L:swwalkgain} $\phi(W^\zeta) = \phi(W)$.  That is, the sign of $\phi(e_i)$ in the sum is still determined by the number of preceding incoherent vertices.

\subsubsection{Sign-cycle gains}\

Consider a positive circle $C$.  Its gain $\phi(C)$ is the sum of its edge gains after sign switching and reorientation so it is an all-positive sign cycle.  The only uncertainty is the sign of $\phi(C)$, which depends on the choice of direction around $C$.  There is not much else to say about it.

A handcuff $C$ is a different story.  First we need a more elementary concept.

A \emph{rooted negative figure} $(C_1,v)$ is a negative figure $C_1$ (a circle or half edge) with a distinguished vertex $v$, the \emph{root}.  
If $C_1$ is a circle, to compute its gain we orient it as an excycle rooted at $v$, i.e., so that $\tau(v,e)=-1$ for both edges at $v$.  Then the \emph{gain of $(C_1,v)$}, written $\phi(C_1,v)$, is the gain of a minimal walk around $C$ from $v$ to $v$.  By Lemmas \ref{L:swwalkgain} and \ref{L:revwalkgain} and since $C$ is negative, $\phi(C_1,v)$ is independent of the particular orientation and choice of direction, except that it negates if $\tau(v,e)=+1$ for the edges at $v$.  

If $C_1=\{e\}$, consisting of a half edge $e$ at vertex $v$, orient it as an excycle, that is with $\tau(v,e) = -1$; the gain $\phi(C_1,v)$ is the gain of the excycle $\{e\}$ considered as an oriented ultrawalk, which by \eqref{E:walkgainsum2} is $\phi(e)$ since $(C_1,v)$ is an excycle.

Now, orient the edges of the handcuff $C$ so it is a sign cycle.  If its two negative figures are circles, choose a circuit walk $W$ around $C$ (which will necessarily be coherent).  The gain of $C$ is $\phi(C) := \phi(W)$.  If one negative figure in $C$ is a circle and the other is a half edge $e$, we choose $W$ to be a circuit ultrawalk that begins and ends with $e$.  If both negative figures are half edges, we choose $W$ to be a circuit ultrawalk that begins at one half edge and ends at the other (this type is exceptional in that the connecting path is traversed only once).  

For a loose edge $e$, its gain as a sign circuit is its gain as an edge.

In the following proposition we compute the gain of a sign circuit.  

\begin{prop}\label{L:gainsum}
The gain of a sign circuit $C$ is well defined up to negation and choice of circuit walk.  
For a positive circle $C$, it is the sum of the edge gains when $C$ is oriented so every vertex is coherent.  
If $C$ is a handcuff with rooted negative figures $(C_1,u_1)$ and $(C_2,u_2)$ where the roots are the endpoints of the connecting path $P$, and $P$ is oriented out of $u_1$, then 
$$\phi(C) = \phi(C_1,u_1) - \phi(C_2,u_2) - 2 \phi(P_{u_1u_2})$$ 
(or its negative) if both negative figures are circles,
$$\phi(C) = 2\phi(C_1,u_1) - \phi(C_2,u_2) - 2 \phi(P_{u_1u_2})$$ 
(or its negative) if $C_1$ is a half edge and $C_2$ is a circle, and
$$\phi(C) = \phi(C_1,u_1) - \phi(C_2,u_2) - \phi(P_{u_1u_2})$$ 
(or its negative) if both negative figures are half edges.
\end{prop}

\begin{proof}
If $C$ is a positive circle, this is obvious.  
Thus, assume $C$ is a handcuff with connecting path $P=P_{u_1u_2}$ from $u_1$ to $u_2$.  By sign switching make $P$ positive.  Note that $\phi(C_i)$ for a negative circle is computed by taking a walk on it from $u_i$ to $u_i$.  For a half edge $\phi(C_i)$ is computed by taking a walk from $u_i$.

First, consider the case where $C_1$ and $C_2$ are circles.  
We apply Lemma \ref{L:sumwalkgains} to $W = C_1PC_2P\inv$ but we have to be careful about signs.  
We get 
\begin{align*}
\phi(W) &= \phi(C_1) + \s(C_1)\phi(P) + \s(C_1P)\phi(C_2) + \s(C_1PC_2)\phi(P\inv) 
\\&= \phi(C_1) - \phi(P) - \phi(C_2) + \phi(P\inv)
\\&= \phi(C_1) - \phi(C_2) - 2\phi(P)
\end{align*}
because  $\phi(P\inv) = -\phi(P)$ by Lemma \ref{L:revwalkgain}.

If we interchanged the roles of $C_1$ and $C_2$, taking the walk $W' = C_2P\inv C_1P$, we would get the negated gain $\phi(C_2) - \phi(P\inv) - \phi(C_1) + \phi(P) = \phi(C_2) - \phi(C_1) + 2\phi(P)$; that explains the ambiguity of negation.

Second, consider the case where $C_1=\{e_1\}$ for a half edge $e_1$ and $C_2$ is a circle.  
Take the circuit ultrawalk $W = C_1u_1Pu_2C_2u_2P\inv u_1C_1$.  
We initially walk on $C_1$ from $e_1$ to $u_1$, which is $C_1\inv$ because $C_1$ is an excycle for computing $\phi(C_1)$.
The calculation gives 
\begin{align*}
\phi(W) &= \phi(C_1\inv) + \s(C_1)\phi(P) + \s(C_1P)\phi(C_2) 
\\&\quad + \s(C_1PC_2)\phi(P\inv) + \s(C_1PC_2P\inv)\phi(C_1)
\\&= -\s(C_1)\phi(C_1) - \phi(P) - \phi(C_2) + \phi(P\inv) + \phi(C_1)
\intertext{by Lemma \ref{L:revwalkgain} applied to $C_1$,}
&= 2\phi(C_1) - \phi(C_2) - 2 \phi(P).
\end{align*}

Finally, consider the case where $C_1 = \{e_1\}$ and $C_2 = \{e_2\}$ are both half-edge figures.  
Take the circuit ultrawalk $W = e_1u_1Pu_2e_2$.  
Then 
\begin{align*}
\phi(W) 
&= \phi(C_1\inv) + \s(C_1)\phi(P) + \s(C_1P)\phi(C_2) 
\\&= \phi(C_1) - \phi(P) - \phi(C_2) .
\end{align*}

Independence of the choice of circuit walk follows from the formulas for $\phi(C)$.
\end{proof}

Now we can make the key definition.

\begin{defn}[Hyperbalance]\label{D:neutral}
A sign circuit is \emph{neutral} if its gain is $0$.  
An arbitrary edge set or subgraph is \emph{hyperbalanced} if every sign circuit in it is neutral; otherwise it is \emph{hyperfrustrated}.  
\end{defn}

Neutrality is the fundamental fact about an edge set.  It is independent of orientation, because by Lemmas \ref{L:swwalkgain} and \ref{L:revwalkgain} neutrality of a sign circuit is independent of orientation.  The list of neutral sign circuits is the essential datum upon which we base the definition of the matroid of a gain signed graph.

\subsubsection{Gain switching}\

Gains can be switched by a gain-switching function, briefly \emph{gain switcher}, which is a function $\theta: V \to \K^+$ (until Section \ref{sec:abstract}).  The gain function $\phi$ switches to $\phi^\theta$ defined on an oriented edge $e=e_{vw}$ by 
\begin{align*}
\phi^\theta(e) &:= \tau(v,e)\theta(v) + \phi(e) + \tau(w,e)\theta(w)
\\&	= \phi(e) + \tau(v,e)[\theta(v) - \s(e)\theta(w)].
\end{align*}
For example, if $\theta$ is a constant, then $\phi^\theta(e) = \phi(e) + \tau(v,e)(1-\s(e))\theta$.  
For a half edge $e=e_v$, 
$$
\phi^\theta(e) := \tau(v,e)\theta(v) + \phi(e).
$$
Gain switching has no effect on a loose edge.

Very similar is the concept of a \emph{gain potential} for an edge set $S$, which is a mapping $\bar\theta: V \to \K^+$ such that $$\phi(e_{vw}) = \tau(v,e)\bar\theta(v) + \tau(w,e)\bar\theta(w)$$ for every link or loop in $S$, $\phi(e_v) = \tau(v,e)\bar\theta(v)$ for every half edge in $S$, and $\phi(e_\eset) = 0$ for every loose edge in $S$.  A gain potential for $S$ is the negative of a gain switcher that switches $S$ to all neutral gains.  For a gain switcher $\theta$ we define 
\begin{align*}
E(\theta) &:= (\phi^\theta)\inv(0) 
\\&\,= \{ e_{vw} \in E: \phi(e_{vw}) = -\tau(v,e)\theta(v) - \tau(w,e)\theta(w) \} 
\\&\qquad \cup \{ e_v \in E: \phi(e_v) = -\tau(v,e)\theta(v) \} 
\\&\qquad \cup E^{00},
\end{align*}
i.e., it is the largest edge set such that $-\theta$ is a gain potential for these edges.  
Combining gain and sign potentials, $E(\theta, \zeta) := E(\theta) \cap E(\zeta)$, which is the set of edges with sign determined by $\zeta$ and gain by $\theta$.

\begin{lem}\label{L:swgainswalk}
Switching the gains does not change the gain sum of a positive closed walk or an ultrawalk with a half edge at each end.  More generally, switching the gains on a walk or ultrawalk $W$ from $u_0$ to $u_l$ ($l\geq0$) gives switched gain 
\begin{align*}
\phi^\theta(u_0e_1\cdots u_l) &= -\theta(u_0) + \phi(W) + \s(W)\theta(u_l),
\\
\phi^\theta(e_0u_0e_1 \cdots u_l) &= -\theta(u_0) + \phi(W),
\\
\phi^\theta(u_0e_1 \cdots u_le_{l+1}) &= \phi(W) + \s(W_{0l})\theta(u_l),
\\
\phi^\theta(e_0u_0e_1 \cdots u_le_{l+1}) &= \phi(W_{0l}),
\end{align*}
\end{lem}

\begin{proof}
We check by a calculation, which we present for an ultrawalk with a half edge at each end:
\begin{align*}
\phi^\theta(W) &= \tau(u_0,e_0)[\phi(e_0) + \tau(u_0,e_0)\theta(u_0)] 
\\&\quad - \sum_{i=1}^l [\tau(u_{i-1},e_i)\theta(u_{i-1}) + \phi(e_i) + \tau(u_i,e_i)\theta(u_i)] \s(W_{0,i-1}) \tau(u_{i-1},e_i) 
\\&\quad - [\tau(u_{l},e_{l+1})\theta(u_{l}) + \phi(e_{l+1})] \s(W_{0l}) \tau(u_{l},e_{l+1})
\\&= \theta(u_0) - \sum_{i=1}^l [\theta(u_{i-1}) + \tau(u_i,e_i)\tau(u_{i-1},e_i) \theta(u_i)] \s(W_{0,i-1})
\\&\quad - \theta(u_{l})\s(W_{0l}) + \phi(W) 
\\&= \theta(u_0) - \sum_{i=1}^l \theta(u_{i-1}) \s(W_{0,i-1})
	+ \sum_{i=1}^l \theta(u_i) \s(W_{0i}) - \theta(u_{l})\s(W_{0l})
	+ \phi(W) 
\\&= [\theta(u_0) - \theta(u_{l})\s(W_{0l})] - \theta(u_0) + \theta(u_l) \s(W_{0l}) + \phi(W).
\end{align*}
The two terms in square brackets apply when there is an initial half edge $e_0$ (the first term) or a terminal half edge $e_{l+1}$ (the second term).  The stated formulas follow from this.

Evidently, $\phi^\theta(W)$ equals $\phi(W)$ when $W$ is closed and positive and also when it begins and ends with a half edge.
\end{proof}

The effect of switching on the gain of a walk does not depend on orientation.

\begin{thm}\label{P:swhyperbal}
Switching of signs and gains preserves the property of hyperbalance or hyperunbalance of every edge set.
\end{thm}

\begin{proof}
For sign switching this is obvious.  The gain sum of a sign circuit equals the gain sum of a circuit walk on it, and by Lemma \ref{L:swgainswalk} switching preserves that sum, so a neutral or non-neutral sign circuit remains neutral or non-neutral, respectively.
\end{proof}

For stating the next lemma we define a \emph{pseudotree} to be a tree or a tree with an attached half edge.  A  \emph{pseudoforest} is a graph whose components are pseudotrees.  If $T$ is a pseudotree with a half edge $e$ at vertex $v$, then $T_{re}$ denotes the path in $T$ from $r$ ending with $e$.

\begin{lem}\label{L:0gaintree}
Let $\Ups$ be a gain signed graph and $T$ a pseudoforest.  The gains can be switched to $\Ups^\theta$ in which $T$ has all gains $0$.  

For connected $T$, the gain-switching function is given by $\theta(v) := \s(T_{rv})[\theta_0 + \phi(T_{rv})]$, where $r$ is an arbitrarily fixed root vertex.  If $T$ is a tree, $\theta_0$ is any element of $\K$.  If $T$ has a half edge $e$ at vertex $v$, then $\theta_0 = \phi(T_{re})$.
\end{lem}

\begin{proof}
We may assume $\Ups$ is connected and the root $r$ is fixed.  We prove that $\theta$ defined in the lemma is a switching function that makes all gains $0$.

Consider an edge $e_l=vw\in T$ whose endpoint farther from the root is $w$, and let $T_{rw} = u_0e_{01}\cdots u_{i-1}e_{i-1,l}u_l$ in $T$, where $r=u_0$, $e = e_{l-1,l}$, $v=u_{l-1}$, and $w=u_l$).  The switched gain is given by 
\begin{align*}
0 = \phi^\theta(e) = \tau(u_{l-1},e)\theta(u_{l-1}) + \phi(e) + \tau(u_l,e)\theta(u_l), 
\end{align*}
from which it follows that 
\begin{align*}
\s(T_{rw})\theta(w) &= \s(T_{ru_{l-1}})\theta(u_{l-1}) + \s(T_{ru_{l-1}})\tau(u_{l-1},e_{l-1,l})\phi(e_{l-1,l})
\\&= \theta(r) + \sum_{i=1}^l \s(T_{ru_{i-1}})\tau(u_{i-1},e_{i-1,i})\phi(e_{i-1,i})
\\&= \theta(r) - \phi(T_{rw})
\end{align*}
by \eqref{E:walkgainsum}.  This gives the value of $\theta(w)$, with an arbitrary constant $\theta(r)$ that we call $\theta_0$, provided there is no half edge in $T$.

If there is a half edge $e$ at vertex $v$, its switched gain is $0 = \phi^\theta(e) = \tau(v,e)\theta(v) + \phi(e),$ hence $\theta(v) = -\tau(v,e)\phi(e).$  Inserting this into the general formula for $\theta(v)$ gives the value of $\theta_0$:
$$
\theta_0 = \phi(T_{rv}) - \s(T_{rv})\tau(v,e)\phi(e) = \phi(T_{re})
$$
by \eqref{E:ultrawalkgainsum}.
\end{proof}

It is easy to see that, if $T$ has no half edge, any one choice of root gives all possible switching functions by varying $\theta_0$.  A similar proof establishes that the gains on $T$ can be set by gain switching to any desired values; we omit the details.  
We mention that if $T$ is all positive and is oriented as an out-arborescence from $r$, then $\phi(T_{rv})$ equals the sum of the gains of edges in $T_{rv}$.

\begin{thm}\label{L:hyperbalneutral}
The gain signed graph $\Ups$ is hyperbalanced if and only if its gains can be switched so all edges are neutral.
\end{thm}

\begin{proof}
Sufficiency is trivial.  For necessity, we may assume $\Ups$ is connected and, by suitably switching signs, has an all-positive spanning tree $T$.  By reorientation make all negative edges (including half edges) introverted.  By gain-switching as in Lemma \ref{L:0gaintree} we can make all tree gains $0$.  

First, we show all positive edges have gain $0$.  Let $e$ be such an edge, not in $T$.  The unique circle $C_T(e) \subseteq T \cup \{e\}$ is all positive, hence a sign circuit, hence it has gain $0$.  From the definition, $\phi(C_T(e)) = \pm\phi(e)$; that implies $\phi(e)=0$.

Next, we show all negative links and loops have the same gain.  Two such edges, $e$ and $f$, form circles $C_T(e)$ and $C_T(f)$, whose intersection may be either a path of positive length, or at most a single vertex.  In the former case, $C_T(e) \cup C_T(f)$ is a theta graph in which there is a positive circle $C$ containing both negative edges.  Take a walk $W = u_0eu_1 \cdots u_{i-1}fu_i \cdots u_0$ around $C$; then $\phi(W) = -\tau(u_0,e)\phi(e) - \s(W_{0,i-1})\tau(u_{i-1},f)\phi(f) = \phi(e) - \phi(f)$.  As $C$ is hyperbalanced, $\phi(e)=\phi(f)$.  In the latter case, $T \cup \{e,f\}$ contains a handcuff circuit $C$ in which $e$ and $f$ belong to opposite circles; by Proposition \ref{L:gainsum} a circuit walk around $C$ gives the same conclusion that $\phi(e)=\phi(f)$.  

Consider two half edges, $e$ and $f$.  They are joined by a path $P$ of length $l$ in $T$ with which they make an ultrawalk $W$, which is a circuit walk on the circuit $C = P \cup \{e,f\}$.  As in the previous case, by Proposition \ref{L:gainsum} $\phi(W) = \phi(e) - \phi(f)$ so $\phi(e)=\phi(f)$ by hyperbalance.  That is, all half edges have the same gain.

Now compare a half edge $e$ to a negative link or loop $f$.  $T \cup \{e,f\}$ contains a handcuff circuit $C$ in which $f$ belongs to a negative circle and $\phi(C) = 2\phi(e) - \phi(f) = 0$.  Thus, the gain of a half edge equals half that of a negative link or loop.

Let the common value of $\phi(e)$ for all negative links and loops be $\phi_0$ and switch by the constant switching function $\theta \equiv \frac12 \phi_0$; or if there are no such edges, let $\phi_1$ be the gain of a half edge and switch by the constant function $\theta \equiv \phi_1$.  This reduces the gains of all negative edges to $0$.

Hyperbalance implies that all loose edges are neutral.  Thus, all edges have gain $0$ after gain switching.

Finally, switch the signs back with the original sign switcher.  This does not change the gains, which remain identically $0$.
\end{proof}

Now that we have three separate operations, two switchings and a reorientation, that preserve essential properties, it is desirable to know how they commute.  
Sign and gain switching, on the other hand, do not commute.  Define an action of $\zeta: V \to \{\pm1\}$ on $\theta: V \to \K^+$ by $\theta^\zeta(v) := \zeta(v)\theta(v)$.

\begin{prop}[Commutation of Switching and Reorientation]\label{P:commut}
The commutation relations among sign switchers $\zeta$, gain switchers $\theta$, and reorientation functions $\rho$ are  $\zeta\theta = \theta^\zeta \zeta$, $\zeta\rho = \rho\zeta$, and $\theta\rho = \rho\theta$.
\end{prop}

\begin{proof}
We leave to the reader the proof that reorientation commutes freely with both switchings.  Hint:  Define a reorientation function $\rho: E \to \{\pm1\}$ such that $\phi^\rho(e) = \rho(e)\phi(e)$ and $\tau^\rho(v,e) = \rho(e)\tau(v,e)$ (assuming $v$ is an endpoint of $e$), i.e., $\rho$ acts on $\phi$ and $\tau$ by multiplication.

Under the action of $\zeta$, the orientation $\tau$ required for gain switching changes to $\tau^\zeta := \zeta\tau$, i.e., $\tau^\zeta(v,e) = \zeta(v)\tau(v,e)$.  The actions of $\theta\zeta$ and $\zeta\theta$ on $\phi$ are given by 
\begin{align*}
\phi^{\theta\zeta}(e_{vw}) := (\phi^\theta)^\zeta(e_{vw}) &= \phi^\theta(e_{vw})
\intertext{because sign switching does not change gains,}
&= \tau(v,e)\theta(v) + \phi(e_{vw}) + \tau(w,e)\theta(w) 
\intertext{for $e = e_{vw}$, while}
\phi^{\zeta\theta}(e_{vw}) := (\phi^\zeta)^\theta(e_{vw}) &= \tau^\zeta(v,e)\theta(v) + \phi^\zeta(e_{vw}) + \tau^\zeta(w,e)\theta(w) 
\\&
= \zeta(v)\tau(v,e)\theta(v) + \phi(e_{vw}) + \zeta(w)\tau(w,e)\theta(w) 
\\&
= \tau(v,e)\theta^\zeta(v) + \phi(e_{vw}) + \tau(w,e)\theta^\zeta(w) 
\\&
= \phi^{\theta^\zeta}(e_{vw}) = \phi^{\theta^\zeta \zeta}(e_{vw}) 
\end{align*}
according to the preceding calculation.  
Note that in $\phi^{\theta^\zeta \zeta}$, $\zeta$ acts on $\tau$ after $\theta^\zeta$ is applied to $\phi$, while in $\phi^{\zeta\theta}$, $\zeta$ acts on $\tau$ before $\theta^\zeta$ is applied.  The action on $\tau$ is why $\zeta\theta = \theta^\zeta \zeta \neq \theta^\zeta$.
\end{proof}

Proposition \ref{P:commut} suggests that the full switching group, which combines both kinds of switching, might be a semidirect product $(\K^+)^V \rtimes \{+1,-1\}^V$, but it is not; the associative law fails, as the reader can verify.

\subsection{The extra point}\label{sec:extra}

In Section \ref{sec:rank} we will define a matroid on the edge set of $\Ups$ that has a one-point extension by an \emph{extra point} $e_\infty$, which is not part of the graph.  
The extension is necessary for a full understanding of the matroid and its canonical hyperplane representation.  
(We call it a point because it is a point in the geometrical interpretation of our matroid; in the projective dual interpretation of Section \ref{sec:projhyp}, where matroid points are affine hyperplanes, $e_\infty$ corresponds to the infinite hyperplane.)  
We write $E_\infty = E \cup \{e_\infty\}.$  
The extra point, not being an edge, does not have a sign or gain.

\sectionpage
\section{The vector model}\label{sec:vector}

Our definition of the matroid of $\Ups$ is modeled on a mental picture of vectors over a field $\K$, so we begin the main work with that picture.  We are in $\K^{1+n}$ with coordinates $x_0,x_1,\ldots,x_n$, the coordinate $x_i$ corresponding to the vertex $v_i$.  The standard unit basis is $\{\e_0,\e_1,\ldots,\e_n\}$.  The $x_0$-coordinate is special: it contains gains.  There is a natural projection $\pi_0: \K^{1+n} \to \K^n$ by deleting the $x_0$-coordinate.

Here is the vector associated to an edge $e=e_{v_iv_j}$.  We assume an orientation $\tau$; reversing the orientation of $e$ negates the vector.
\[
\z(e) = \z_\tau(e) 
:= \phi(e)\e_0 + \tau(v_i,e)\e_i + \tau(v_j,e)\e_j 
= \begin{pmatrix} \phi(e) \\0 \\\vdots \\\tau(v_i,e) \\0 \\\vdots \\\tau(v_j,e) \\0 \\\vdots\end{pmatrix} ,
\]
where the nonzero rows are numbered 0, $i$, $j$, respectively.  This is for a link.  If $e$ is a loop, so $v_i=v_j$, the $\tau$ values are added together in row $i$ of the matrix form.  Thus, a positive loop has $x_i=0$ and a negative loop has $x_i=\pm2$.  
For a half edge $e=e_{v_i}$ we define 
\[
\z(e) = \z_\tau(e) 
:= \phi(e)\e_0 + \tau(v_i,e)\e_i
= \begin{pmatrix} \phi(e) \\0 \\\vdots \\\tau(v_i,e) \\0  \\\vdots\end{pmatrix} .
\]
For a loose edge $e=e_\eset$, 
\[
\z(e) = \z_\tau(e) 
:= \phi(e)\e_0 
= \begin{pmatrix} \phi(e) \\\0\end{pmatrix} 
\]
(which has ambiguous sign if $\phi(e) \neq 0$, but that will not affect any of our theory; for linear algebra with the vector of a loose edge, what matters is only whether the gain is $0$ or not).

This defines a mapping $\z: E \to \K^{1+n}$.  
We call the vectors corresponding to the edges of $\Ups$ the \emph{standard vector representation} of $\Ups$.  Note that $\pi_0\z = \x$, the signed-graph representation mapping of Theorem \ref{T:sgrep}.

The extra point $e_\infty$ corresponds to the vector $\z(e_\infty)=(1,\bf0)$; that correspondence extends $\z$ to a mapping $\z: E_\infty \to \K^{1+n}$.  

The \emph{incidence matrix} of $\Ups$ is the $(1+n)\times|E|$ matrix that has a column for each edge, in which is placed the standard vector $\bfz(e)$.  The \emph{extended incidence matrix} has an extra column $\z(e_\infty)$ for the extra point.
We mention incidence matrices because they have many uses, such as to define flows on the graph, but those uses are outside the scope of this article.

Given a sign switcher $\zeta$, define the diagonal switching matrix $D_\zeta := \left(\begin{matrix} 1 & \0\trans \\ \0 & \Diag(\zeta) \end{matrix}\right)$.  Evidently, this matrix is self-inverse.

\begin{lem}\label{L:vecprops}
Some elementary properties of the vector representation are:
\begin{enumerate}[{\rm(1)}]
\item $\z(e\inv) = -\z(e)$.
\item Switching signs by $\zeta$ multiplies vectors by $D_\zeta$: that is, $\z(e)$ becomes $D_\zeta \z(e)$.
\item Switching gains by $\theta$ changes $\z(e)$ to 
$$[\tau(v_i,e)\theta(v_i) + \tau(v_j,e)\theta(v_j)]\e_0 + \z(e)$$
for a loop or link $e=e_{ij}$ and to $\tau(v_i,e)\theta(v_i)\e_0 + \z(e)$ for a half edge $e=e_i$.  It has no effect on a loose edge.
\end{enumerate}
\end{lem}

Given an orientation $\tau$ of $\S$ and a walk $W$ from $u_0=v_i$ to $u_l=v_j$, we define 
\begin{equation}
\z(W) := - \tau(u_0,e_1)  \sum_{k=1}^l \z(e_k) \prod_{m=1}^{k-1} [-\tau(u_m,e_m)\tau(u_m,e_{m+1})] .
\label{E:vecwalksum}
\end{equation}
For an ultrawalk $W$, add $\tau(u_0,e_0)\z(e_0)$ if $W$ begins with a half edge $e_0$ and add $-\tau(u_l,e_{l+1})\s(W_{0l})\z(e_{l+1})$ if $W$ ends with a half edge $e_{l+1}$.

\begin{lem}\label{L:walkvector}
For an ultrawalk $W = e_0u_0\cdots u_le_{l+1}$ that begins with a half edge $e_0$ at vertex $u_0$ and ends with a half edge $e_{l+1}$ at vertex $u_l$, we have
\begin{align*}
\z(W) &= \z(e_0)\tau(u_0,e_0) - \sum_{k=1}^l \z(e_k) \s(W_{0,k-1}) \tau(u_{k-1},e_k) - \z(e_{l+1})\s(W_{0l})\tau(u_l,e_{l+1})
\\&= \z(e_0)\tau(u_0,e_0) + \sum_{k=1}^l \z(e_k) \s(W_{0,k}) \tau(u_{k},e_k) - \z(e_{l+1})\s(W_{0l})\tau(u_l,e_{l+1}).
\end{align*}
If the ultrawalk begins at vertex $u_0$ (without $e_0$), omit the first term.  If it ends at vertex $u_l$ (without $e_{l+1}$), omit the last term.
\end{lem}

\begin{proof}
The expressions for $\z(W)$ equal the definition for the same reason the two expressions in formula \eqref{E:walkgainsum} are equal.  
\end{proof}

\begin{prop}\label{T:closedwalkvector}
The vector $\z(W)$ of a positive closed walk or an ultrawalk with initial and terminal half edges is unchanged by gain switching.
\end{prop}

\begin{proof}
See Lemma \ref{L:swgainswalk}.
\end{proof}

We can state walk vectors in the same form as we have stated edge vectors.  To emphasize the similarity we describe the walk vectors as if they were the vectors of fictitious edges.

\begin{prop}\label{T:walkvector}
\begin{enumerate}[\rm(I)]
\item For a walk $W$ from $v_i$ to $v_j$, the vector $\z(W)$ equals the vector $\z(e_{ij}) = \phi(W)\e_0 - \e_i + \s(W)\e_j$ of an edge $e_{ij}$ with sign $\s(W)$ and gain $\phi(W)$. 

In particular, for a closed walk $W$ from $v_i$ to $v_i$, the vector $\z(W)$ equals the vector $\phi(W)\e_0 - (1-\s(W))\e_i$ of a loop at $v_i$ with sign $\s(W)$ and gain $\phi(W)$.  

For a circuit walk $W$ on a sign circuit $C$, $\z(W) = \phi(C)\e_0$.

\item For an ultrawalk $W$ from $v_i$ to a terminal half edge, the vector $\z(W)$ equals the vector $\z(e_{i}) = \phi(W)\e_0 - \e_i$ of an introverted half edge $e_{i}$ with gain $\phi(W)$. 

\item For an ultrawalk $W$ from an initial half edge to $v_j$, the vector $\z(W)$ is $\phi(W)\e_0 + \s(W)\e_j$, which is the vector $\z(e_j)$ of a half edge $e_j$ with gain $\phi(W)$ and which is extraverted or introverted depending on the sign of $W$. 

\item For an ultrawalk $W$ that begins and ends with a half edge, the vector $\z(W)$ equals $\phi(W)\e_0$, the vector of a loose edge with gain $\phi(W)$.
\end{enumerate}
\end{prop}

\begin{proof}
We prove part (I).  Let $W=u_0e_1u_1\cdots e_lu_l$ where the vertices are $u_k=v_{i_k}$, and the edges are $e_k = u_{k-1}u_k$.  Then
\begin{align*}
\z(W) &= \sum_{k=1}^l \z(e_k) \s(W_{0,k}) \tau(u_{k},e_k) 
\\&
= \sum_{k=1}^l \big[ \phi(e_k)\e_0 + \tau(u_{k-1},e_k)\e_{i_{k-1}} + \tau(u_k,e_k)\e_{i_k} \big] \s(W_{0,k}) \tau(u_k,e_k) 
\\&
= \sum_{k=1}^l \phi(e_k) \s(W_{0,k}) \tau(u_k,e_k) \e_0 
  \\& \qquad + \sum_{k=1}^l  \s(W_{0,k-1})\tau(u_{k-1},e_k)\tau(u_k,e_k) \e_{i_{k-1}} 
	+ \sum_{k=1}^l \s(e_k) \s(W_{0,k-1}) \e_{i_k} 
\\&
= \phi(W) \e_0 - \sum_{k=1}^l  \s(W_{0,k-1})\s(e_k) \e_{i_{k-1}} 
	+ \sum_{k=1}^l \s(e_k) \s(W_{0,k-1}) \e_{i_k} 
\\&
= \phi(W)\e_0 - \e_{i_0} + \s(W_{0l}) \e_{i_l} 
\end{align*}
because of the definition of $\phi(W)$ in \eqref{E:walkgainsum} and our convention that $\tau(u_0,e_1)=-1$ for a walk beginning at $u_0$.

The proofs of the three other cases are the same except for the extra initial and final terms, which cancel terms in case (I).  The extra term for a terminal half edge $e_{l+1}$ equals $-\s(W)\e_j$.  The extra term for an initial half edge $e_0$ equals $\e_i$.
\end{proof}

\begin{cor}\label{C:neutralcircuit}
A sign circuit is a circuit in the vector model if and only if it has gain $0$.  Otherwise, it is linearly independent
\end{cor}

By ``in the vector model'' we mean the vectors that represent the edges of the sign circuit.

\begin{proof}
A proper subset $S$ of a sign circuit $C$ is independent in the frame matroid $\K(\S)$ and hence the set $\{\pi_0 \z(e): e \in S\}$ of projected vectors is independent.  It follows that $\{ \z(e): e \in S\}$ is independent.

The linear combination of vectors $\pi_0 \z(e)$ for $e \in C$ that yields $\0$ is unique (up to scaling) because the vectors are minimally dependent.  Therefore, the vectors $\z(e)$ can be dependent only if that linear combination has $x_0=0$.  The corollary now follows from Proposition \ref{T:closedwalkvector}.
\end{proof}

This is a partial solution to finding the circuits of the matroid $\M(\Ups)$.  The full answer is complicated, so we turn  in the next section to the simpler question of rank.

\begin{exam}[Edge Points of an Edge Polytope \cite{OH1}]\label{X:edgepolytope}
The example that led us to gains on signed graphs comes from algebra.  For an edge $e_{ij}$ in a graph $\G$ (all of whose edges are links and loops) its edge point is $\e_i+\e_j \in \bbA^n(\bbR)$ (which corresponds to the vector $\z(-e_{ij}) = \e_0+\e_i+\e_j \in \bbR^{1+n}$ of the extraverted orientation of $-e_{ij}$).  The convex hull $P$ of the edge points, known as the \emph{edge polytope} of $\G$, is related to binomial ideals.  Ohsugi and Hibi found that the affine dimension of the edge polytope (assuming $\G$ is connected) is $n-2$ if $\G$ is bipartite and $n-1$ if it is not \cite[Proposition 1.3]{OH1}.  Restated in terms of signed graphs:  the dimension is $n-2$ if $-\G$ is balanced and 1 greater if it is not.  As affine dimension is 1 less than matroid rank, this corresponds exactly to the rank of the frame matroid $\bfF(-\G)$ and should be deducible by matroid theory.  In developing our theory we wanted to generalize to arbitrary oriented signed graphs, where the edge points of this example correspond to extraverted edges.  For a point in affine space its vector has the extra coordinate $x_0=1$, thus the gains are identically 1.  We found it easiest and most enlightening to allow arbitrary gains (as far as possible).  Hence, gain signed graphs.

We give a matroidal proof of Ohsugi and Hibi's Proposition 1.3.  When all edges are negative, the sign-balanced circles are those of even length; thus, the sign circuits are the even circles and the handcuffs with two odd circles.  
Taking gains identically 1 implies that every sign circuit is neutral so the gain signed graph $\Ups=(\G,-1,1)$ (that is, all edges being negative with gain 1) is hyperbalanced.    
Now we anticipate the rank function of Section \ref{sec:rank}.  Let $b(S)$ be the number of bipartite components of an edge set $S$ (considered as a spanning subgraph).  An edge set $S$ has rank $n-b(S)$.  For a connected graph, therefore, $\rk_\Ups(E) = n-1$ if $\G$ is bipartite and $n$ if not.  The rank is the linear dimension of the vectors in $\bbR^n$, but since the edge vectors $\e_i+\e_j$ are contained in the inhomogeneous hyperplane $\sum x_i=2$, their affine dimension is one less; that is, $\dim P = n-2$ if $\G$ is bipartite and $n-1$ otherwise.
\end{exam}

\begin{exam}[Edge Points for a Bidirected Graph]\label{X:affdim}
The edge points of Example \ref{X:edgepolytope} are the columns of the unoriented incidence matrix of the graph.  (To define that matrix, reverse the previous sentence.)  After the edge polytope was introduced, Matsui et al.\ \cite{MH}) considered the analog for an oriented incidence matrix, which is the incidence matrix of an all-positive signed graph.  (To get that matrix, take an all-positive gain signed graph and delete the row of gains.)  An oriented positive edge is an ordinary directed edge.  Our theory enables us to state the affine dimension of the points corresponding to directed edges, which may differ from their linear dimension (see the next example).  In fact, we can state a general theorem for the points obtained from any edges of any bidirected graph.  Again, we anticipate Theorem \ref{T:rank}.

First we examine coherent and incoherent vertices of a closed walk in a bidirected graph.  Let $W = u_0e_1u_1\cdots e_lu_l$ be a closed walk; that is, $u_0=u_l$.  For convenience, define $e_{l+1}=e_1$.  The sign of $W$ is 
$$\s(W) = \prod_{i=1}^l \s(e_i) = \prod_{i=1}^l \big[ -\tau(u_{i-1},e_i)\tau(u_i,e_i) \big] = (-1)^l \prod_{i=1}^l \tau(u_{i-1}) \prod_{i=1}^l \tau(u_i,e_i).$$
Define the incoherence sign for vertex $u_i$ to be $\iota(u_i) = +1$ if $W$ is coherent at $u_i$ and $-1$ if $W$ is incoherent; that includes $u_l=u_0$ with edges $e_l$ and $e_{l+1}=e_1$.  Thus, $\iota(u_i) = -\tau(u_i,e_i)\tau(u_i,e_{i+1})$.  Then the product of all vertex signs of $W$ is 
$$\prod_{i=1}^l \iota(u_i) = \prod_{i=1}^l \big[ -\tau(u_i,e_i)\tau(u_i,e_{i+1}) \big] = (-1)^l \prod_{i=1}^l \tau(u_{i-1}) \prod_{i=1}^l \tau(u_i,e_i) = \s(W).$$
In other words:

\begin{lem}\label{L:incoherentwalk}
The number of incoherent vertices in a closed walk $W$ in a bidirected graph is even if and only if $W$ is positive.
\end{lem}

Now we define poise of a closed walk $W$.  We assign the edges to two sets, $A$ and $B$.  We put $e_1$ into set $A$ and for each edge $e_i$ in one of the sets, we put $e_{i+1}$ into the same set if $u_i$ is coherent and the opposite set if it is incoherent.  This gives a well-defined bipartition of the edges of $W$ into sets $A$ and $B$ if and only if both the number of changes of set, which equals the number of incoherent vertices, is even, i.e., $W$ is positive, and also a repeated edge is assigned to the same set in every appearance in $W$.  We define $W$ to be \emph{poised} if the bipartition is well defined and $|A|=|B|$.  

We also define poise for a walk $W$ from a half edge to a half edge (which may be the same edge).  We put the initial half edge into $A$ and apply the same rule as before.  We get a well-defined bipartition of the edges of $W$ if and only if a repeated edge is assigned to the same set at every appearance.  We say $W$ is \emph{poised} if the bipartition is well defined and $|A|=|B|$.

Now we can state the theorem about edge points.  A sign circuit is poised if its circuit walk is poised; since a sign circuit walk is either positive or begins and ends at a half edge, this definition is independent of the choice of circuit walk.  

\begin{thm}[Dimension of Bidirected Edge Points]\label{T:affdim}
In a bidirected graph $\vec\S$ let $S \subseteq E$.  The affine dimension of the point set $\x(S)$ in $\bbA^d(\bbR)$ is $n-b_\S(S)$ if every sign circuit in $S$ is poised and $n-b_\S(S)+1$ if not.
\end{thm}

\begin{proof}
An affine point $\bfa \in \bbA^n(\bbR)$ corresponds to the vector $\e_0+\bfa$ in $\bbR^{1+n}$.  It is well known that the affine dimension of a set of affine points equals the linear dimension of the corresponding vectors.  Thus, we are assigning gain $1$ to every edge of $\vec\S$, forming a gain signed graph $\Ups$.  A sign circuit with these gains is neutral if and only if it is poised.  Theorem \ref{T:rank} then gives the rank of an edge set:  $\rk_\Ups(S) := n - b_\S(S) + \delta_\Ups(S)$, where $\delta_\Ups(S) = 0$ if every sign circuit in $S$ is poised, and otherwise is $1$.  As we observe immediately after Theorem \ref{T:rank}, the rank of $S$ equals the dimension of $\z(S)$, which equals the affine dimension of $\x(S)$ by the relation between affine points and their corresponding vectors.
\end{proof}
\end{exam}

\begin{exam}[Arc Adjacency Polytope \cite{OH1, CDK}]\label{adjpoly1}
As a special case of the previous example, suppose all edges in $\S$ are positive; i.e., $\vec\S$ is a directed graph $\vec\G$.  The sign circuits are the circles.  A circle is poised if and only if it has equal numbers of directed edges in each direction.  The convex hull of the set $\x(E)$ of affine points of a symmetric digraph (where for every arc there is also the opposite arc) has been called its symmetric edge polytope \cite{MH} and also its adjacency polytope (e.g., \cite{CDK}); as in \cite{CDK} we generalize this to the \emph{arc polytope} of any directed graph (an arc being a directed edge).  Let $c(\G)$ denote the number of connected components of $\vec\G$.

\begin{cor}[Dimension of Arc Polytope]\label{C:dimadj1}
For a directed graph $\vec\G$, the affine dimension of its arc polytope is $n-c(\G)$ if every circle in $\vec\G$ is poised and $n-c(\G)+1$ if not.
\end{cor}
\end{exam}

\begin{exam}[Double Arc Adjacency Polytope \cite{CM}]\label{adjpoly2}
There is a second polytope that has been called the adjacency polytope.  Again suppose we have a directed graph $\vec\G = (V,E)$.  We represent an arc $(v_i,v_j)$ by the point $(\e_i,\e_j) \in \bbA^{2n}(\bbR)$, thus in twice the dimension of the adjacency polytope; we call the convex hull of these \emph{double arc points} the \emph{double adjacency polytope} of $\vec\G$.  Chen and Mehta \cite{CM} introduced this polytope for the special case of a symmetric digraph (which they regarded as an undirected graph).  Note that if we represent $(v_i,v_j)$ by $f(v_i,v_j) = (-\e_i,\e_j)$, nothing essential changes in the structure of the set $f(E)$: dimension and convexity remain the same.  We use $f$ as the representation for the rest of this example.  

We determine the dimension of a set of double arc points by a new digraph $\tilde\G$, which is a vertex doubling of $\vec\G$.  The vertex set is $\tilde{V} = V_+ \cup V_-$, where $V_+$ and $V_-$ are disjoint copies of $V$ whose vertices corresponding to $v \in V$ are respectively $v_+$ and $v_-$, and for each arc $(v_i,v_j)$ we create an arc $(v_{i-},v_{j+})$ in $\tilde\G$.  Note that every circle in $\tilde\G$ is poised because its vertices alternate between $V_+$ and $V_-$, so its arcs reverse direction at every vertex.  Therefore, Corollary \ref{C:dimadj1} applies and we have the following description of dimension:

\begin{cor}[Dimension of Double Arc Polytope]\label{C:dimadj2}
The affine dimension of the double arc polytope of a digraph $\vec\G$ equals $2n-c(\tilde\G)$, which is also its linear dimension.
\end{cor}

\begin{proof}[Completion of Proof]
The linear dimension is the same as the affine dimension because it equals the rank of the incidence matrix of $\vec\G$, which is $2n-c(\tilde\G)$.
\end{proof}

Of course, this corollary can be applied to any subset of the arcs.

A special case of particular interest is that in which $\vec\G$ is symmetric and has all loops $(v_i,v_i)$.  Then $c(\tilde\G) = c(\G)$, which is a nice simplification.
\end{exam}

\sectionpage
\section{Matroid: rank}\label{sec:rank}

We want a combinatorial description of the matroid obtained by representing each edge of $\Ups$ by a vector as in Section \ref{sec:vector}.  
The problem is to describe the linear dependence matroid of the vectors without reference to linear algebra.  Thus, we want combinatorial formulas for the rank function, the circuits, the independent sets and bases, the closed sets (or flats), and the coatoms of the lattice of flats, whose complements are the matroid cocircuits.  We call this the \emph{matroid of $\Ups$} and denote it by $\M(\Ups)$.  We base the matroid on its rank function.

The matroid has a natural one-point extension, which we denote by $\M_\infty(\Ups)=\M(\Ups)\cup\{e_\infty\}$, to the extra point.  
Whereas the ground set of $\M(\Ups)$ is $E=E(\Ups)$, that of $\M_\infty(\Ups)$ is $E_\infty = E\cup\{e_\infty\}$.  We call $\M_\infty(\Ups)$ the \emph{extended matroid} of $\Ups$.  
It is a one-point coextension of the frame matroid $\bfF(\Sigma)$; that is, $\M_\infty(\Ups)/e_\infty=\bfF(\Sigma)$.  (One purpose of the extra point is to implement this property of $\M(\Ups)$.)  
A subset of $E_\infty$ is defined to be hyperfrustrated if it contains $e_\infty$; if not, it is a subset of $E$ whose treatment is as we have already described.  Hence, in the matroid the extra point behaves like a non-neutral loose edge (and we treat it as such in proofs), although its significance is different in that it implements the coextension of $\bfF(\Sigma)$.

The purpose of the extra point will become clearer in Section \ref{sec:closure}.  For the present we only mention that, if all edges are positive, so that $\Ups$ is a gain graph $\Phi=(\G,\phi)$, then $\M(\Ups)$ is the lift matroid $\bfL(\Phi)$ and $\M_\infty(\Ups)$ is the extended lift matroid $\bfL_\infty(\Phi)$ (formerly written $\bfL_0$) of \cite[Section 3]{BG2}.

\begin{defn}\label{D:rank}
The \emph{rank} of an extended edge set $S \subseteq E_\infty$ is
\begin{equation*}
\rk_\Ups(S) := n - b_\S(S) + \delta_\Ups(S),
\end{equation*}
{where}
\begin{equation*}
\delta_\Ups(S) :=\begin{cases}
	0	&\text{ if $S$ is hyperbalanced,} \\
	1	&\text{ if $S$ is hyperfrustrated.}
	\end{cases}
\label{E:rank}
\end{equation*}
We define $\rk(\Ups)=\rk_\Ups(E)$.
\end{defn}

For an extended edge set $S$, $\z(S)$ denotes the multiset of vectors representing $S$, i.e., $\z(S) = \{ \z(e) : e \in S \}$.  The dimension of an arbitrary (multi)set of vectors means the dimension of its linear span.

\begin{thm}\label{T:rank}
The function $\rk_\Ups$ defines a matroid on ground set $E_\infty$ such that 
$$
\rk_\Ups(S) = \dim \z(S)
$$
for every edge set $S$.  
\end{thm}

That matroid on ground set $E$ is the matroid $\M(\Ups)$, and its extension to $E_\infty$ is the extended matroid $\M_\infty(\Ups)$.
Thus, the vector model, i.e., the function $\z: E \to \K^{1+n}$, is a vector representation of this matroid.

\begin{proof}
The first step is to dispose of the extra point.  Observe that $\z(e_\infty)$ is a nonzero scalar multiple of $\z(e)$ for any non-neutral loose edge $e$.  Therefore, for the matroid we can treat $e_\infty$ as another non-neutral loose edge.

We prove that $\dim\z(S)$ satisfies Equation \eqref{E:rank}.  Then the matroid exists, because $\dim$ is a matroid rank function, and it is represented by $\z$.  

Because projection cannot increase dimension, for every edge set $S$, 
\begin{equation}
\label{E:dimred}
\dim\z(S) \geq \dim\pi_0\z(S) = \rk_\S(S) = n - b_\S(S)
\end{equation}
by \cite[Theorems 8B.1 and 5.1]{SG}.

\emph{Case 1. $S$ is hyperbalanced.}  

By the definition of hyperbalance, every sign circuit has gain sum 0.  Thus, by Corollary \ref{C:neutralcircuit} for every sign circuit $C \subseteq S$, $\z(C)$ is a circuit in the vector model.  It follows that $\z(R)$ is dependent for every subset $R \subseteq S$ that is dependent in the frame matroid $\bfF(\S)$.  That and \eqref{E:dimred} imply that the vector matroid of $\z(S)$ is the same as the frame matroid $\bfF(\S|S)$, so $\dim\z(S) = \dim\pi_0\z(S) = \rk_\S(S) = n - b_\S(S)$.

\emph{Case 2.  $S$ is not hyperbalanced.}
By definition, there exists a sign circuit $C \subseteq S$ with gain sum nonzero. By Corollary \ref{C:neutralcircuit}, $\z(C)$ is independent.

Suppose first that $C=\{e\}$, so $e$ is a positive loop or a loose edge.  Then $\z(e) = \phi(e)\e_0 \neq \0$.  Let $B$ be a basis for $S$ in $\bfF(\S)$.  No nontrivial linear combination of the vectors in $\pi_0\z(B)$ equals $\0$, so there is no way to express $\z(e)$ as a linear combination of $\z(B)$.  It follows that $\dim\z(B\cup\{e\}) = \dim\pi_0\z(B) + 1 = \rk_\S(B) + 1 = n - b_\S(S) + 1$.

Now consider the general case, $|C|>1$.  
Let $e \in C$ and $T=C \setm e$.  Extend $T$ to a basis $B$ of $\bfF(\S|S)$; thus, $\pi_0\z(B)$ is independent so $\z(B)$ is independent.  We prove $\z(B \cup \{e\})$ is independent.

Suppose by way of contradiction that it is dependent. 
Then there exists a unique circuit $\z(A\cup \{e\})$ in $\z(B \cup \{e\})$.  The projection $\pi_0\z(A\cup \{e\})$ is therefore dependent, which means that $A\cup \{e\}$ is dependent in $\bfF(\S)$.  It follows that $A\cup \{e\}$ contains the unique sign circuit in $B \cup \{e\}$, i.e., $C \subseteq A\cup \{e\}$.  But equality cannot hold because $\z(C)$ is independent; hence, $A \subset T$.

Because $\z(A\cup \{e\})$ is a minimal dependent set, there is a unique linear combination $\z(e) = \sum_{f \in A} \alpha_f \z(f)$, where all $\alpha_f \neq 0$.  Because $C$ is a sign circuit, there is a unique linear combination $\pi_0\z(e) = \sum_{g \in T} \beta_g \pi_0\z(g)$, where all $\beta_g \neq 0$.  So, $\sum_{f \in A} \alpha_f \pi_0\z(f) = \pi_0\z(e) = \sum_{g \in T} \beta_g \pi_0\z(g)$.  Now, $A, T \subseteq B$ and $\pi_0\z(B)$ is an independent set in $\K^n$; consequently $A=T$ and all $\alpha_f = \beta_f$.  However, we saw that $A \supset T$.  This contradiction proves that $\z(B \cup \{e\})$ is independent, and since $|B| = \rk_\S(S)$, we conclude that $\dim \z(B \cup \{e\}) = \rk_s(S) + 1 = n - b_\S(S) + 1$.  

Clearly, $\dim \z(S) \geq \dim \z(B \cup \{e\}) = n - b_\S(S) + 1$.  On the other hand, $\dim \z(S) \leq \dim \pi_0\z(S) + 1 = \rk_\S(S) + 1 = \rk_\S(B) + 1 = n - b_\S(S) + 1$.  That proves the formula for $\dim \z(S)$ in Case 2.
\end{proof}

\sectionpage
\section{Matroid: closure and flats}\label{sec:closure}

Our next mission is to characterize the closed sets of a matroid of a gain signed graph.  Recall that for a set $S$ of elements of a matroid, the closure of $S$ is $\{e: \rk_\Ups(S\cup\{e\})=\rk_\Ups(S)\}$.  
Equivalently, the closure of $S$ is $S\cup\{e: S\cup\{e\}\mbox{ contains a hypercircuit containing $e$}\}$, but we do not use this characterization since we have not yet found the hypercircuits.  

In the matroids $\M(\Ups)$ and $\M_\infty(\Ups)$ the closure of $S$ is called the {\em hyperclosure}, written $\clos_\Ups(S)$ in $\M(\Ups)$ and $\clos_\infty$ in $\M_\infty(\Ups)$.  
To define it we use the \emph{gain balance-closure} operator in $\Ups$, $\bcl_\Ups$, defined by 
$$
\bcl_\Ups(S) := S \cup \{e \notin S: \exists\ \text{a neutral sign circuit $C$ such that } e \in C \subseteq S \cup \{e\} \}.
$$
Note that $\bcl_\Ups(S) \subseteq \clos_\S(S)$.  Also, $\bcl_\Ups(S) \supseteq E^{00}$ since a loose edge is a sign circuit.   
It is easy to prove using a gain potential for $S$ and Theorem \ref{L:hyperbalneutral} that:

\begin{lem}\label{L:gbalclosure}
Let $S \subseteq E$.  If $S$ is hyperbalanced, then $\bcl_\Ups(S)$ is also hyperbalanced.
\end{lem}

\begin{thm}\label{T:closure}
Let $S \subseteq E_\infty$.  If $S$ is hyperbalanced, 
\begin{equation}
\begin{aligned}
\clos_\Ups(S) ,
&= \bcl_\Ups(S)
\\&= [E(\theta){:}U_\S(S)] \cup [E(\theta,\zeta){:}\pib(S)] \cup E^{00}.
\end{aligned}
\label{E:closurehbal}
\end{equation}
where $\zeta$ is a sign potential for $S{:}U_\S(S)^c$ and $-\theta$ is a gain potential for $S$.  The closure is hyperbalanced and is the same in $\M(\Ups)$ and $\M_\infty(\Ups)$.

If $S$ is hyperfrustrated, then in $\M(\Ups)$ 
\begin{equation}
\begin{aligned}
\clos_\Ups(S) 
&= \clos_\S(S)
\\&= [E{:}U_\S(S)] \cup [E(\zeta){:}\pib(S)] \cup E^0,
\end{aligned}
\label{E:closure}
\end{equation}
where $\zeta$ is a sign switcher for $S{:}U_\S(S)^c$.  The closure $\clos_\infty(S)$ in $\M_\infty(\Ups)$ is the same with the addition of $e_\infty$.
\end{thm}

Note that if $X \in \pib(S)$ and $E{:}X$ is sign balanced, then $E(\theta,\zeta){:}X = E(\theta){:}X$ in the first part and $E(\zeta){:}X = E{:}X$ in the second part so $\zeta$ is not necessary for that component of $S$.

\begin{proof}
Recall that we can treat $e_\infty$ as a non-neutral loose edge, so it need not be considered separately.

Let us assume $S$ is hyperfrustrated, so that $\rk_\Ups(S) = n - b_\S(S) + 1$.  
Let $\zeta$ be a sign switcher that switches $S{:}U_\S(S)^c$ to all positive.  
Write $A :=  [E{:}U_\S(S)] \cup [E(\zeta){:}\pib(S)] \cup E^0$.  
It is clear that $S \subseteq A$; thus $\clos_\Ups(S) \subseteq \clos_\Ups(A)$ and $\rk_\Ups(S) \leq \rk_\Ups(A)$.  We want to show that $A$ is closed and has the same rank as $S$.

For the rank, since $A$ is hyperfrustrated, $\rk_\Ups(A) = n - b_\S(A) + 1$, so we should prove $b_\S(A) = b_\S(S)$.  Consider $B \in \pib(S)$.  Then $E(\zeta){:}B$ is sign balanced, and it is connected because $S{:}B$ is connected.  Therefore, $B$ is contained in a block of $\pib(A)$.  But no edge of $A$ connects two different sign-balanced components of $S$, so $B\in\pib(A)$.  This proves $b_\S(A) \geq b_\S(S)$, so that $\rk_\Ups(A) \leq \rk_\Ups(S)$.  That implies equality, so $A \subseteq \clos_\Ups(S)$ and also $\pib(A) = \pib(S)$.

Now we prove that an edge $e \notin A$ is also not in $\clos_\Ups(S)$.  It is sufficient to prove that $b_\S(A \cup \{e\}) < b_\S(A)$, since both $A$ and $A\cup\{e\}$ are hyperfrustrated.  If $e$ joins a sign-balanced component of $A$ to another component, it reduces $b_\S$.  If it joins vertices in a sign-balanced component $A{:}B$ or is a half edge in $E{:}B$, then $A{:}B \supseteq E(\zeta){:}B$ $\implies$ $e \notin E(\zeta)$ $\implies$ $(A \cup \{e\}){:}B$ is sign-unbalanced; then adding $e$ also reduces $b_\S$.   These are the only possibilities, because $E{:}U_\S(S) \cup E^0 \subseteq A$.   That completes the proof that $A = \clos_\Ups(S)$.

If $S$ is hyperbalanced, then we may restrict attention to $\Ups|E(\theta)$ for some gain potential $-\theta$ and either apply the same reasoning as in the previous case or simply appeal to Lemma \ref{L:sgmatroid}.  The closure is hyperbalanced because it is contained in $E(\theta)$.
\end{proof}

\begin{thm}\label{T:closed}
The closed sets of $\M(\Ups)$ are those of the forms
\begin{equation}
[E{:}U] \cup [E(\zeta){:}\pi] \cup E^0,
\label{E:closed}
\end{equation}
where $U \subseteq V$, $\pi$ partitions $U^c$, and $\zeta$ is a sign function on $U^c$, and 
\begin{equation}
[E(\theta){:}U] \cup [E(\theta,\zeta){:}\pi] \cup E^{00},
\label{E:closedhbal}
\end{equation}
where $\theta$ is a gain switcher.  
A closed set of type \eqref{E:closedhbal} is always hyperbalanced, and every hyperbalanced closed set has the form \eqref{E:closedhbal}.

The closed sets of $\M_\infty(\Ups)$ are the same, if the set is hyperbalanced, but the same with the addition of $e_\infty$ if the set is hyperfrustrated.
\end{thm}

The first kind of closed set, \eqref{E:closed}, is the same as the second, hence hyperbalanced and redundant, if $\Ups$ is hyperbalanced, but it is usually not hyperbalanced if $\Ups$ is not hyperbalanced.  It will be hyperbalanced if and only if all loose edges are neutral and $E{:}U$ is hyperbalanced.

Theorems \ref{T:closure} and \ref{T:closed} demonstrate one function of the extra point $e_\infty$: it expresses hyperbalance in purely matroidal terms because an edge set $S \subseteq E$ is hyperbalanced if and only if its closure in $\M_\infty(\Ups)$ does not contain $e_\infty$.

\begin{proof}
Theorem \ref{T:closure} implies that every closed set has one of these two forms.  We prove the converse.  

Let $A$ be the set in \eqref{E:closed}.  The partition $\pi(A)$ refines $\pi \cup \{U\}$.  Thus, if $B_A \in \pi(A)$ is contained in some $B \in \pi$, we have $E(\zeta){:}B_A \subseteq E(\zeta){:}B \subseteq A$.  If $B_A \subseteq U$, then $E{:}B_A \subseteq E{:}U \subseteq A$.  The conclusion is that $\clos_\Ups(A) = [E{:}U_\S(A)] \cup [E(\zeta){:}\pi(A)] \cup E^0 \subseteq A$, from which equality is obvious.

In type \eqref{E:closedhbal}, we can restrict attention to $E(\theta)$, which is hyperbalanced, so its matroid is that of the signed graph $\S(\theta)$.  Then $A$ is closed by Lemma \ref{L:sgmatroid}.  An alternative is to repeat the proof for type \eqref{E:closed} with $E(\theta)$ in place of $E$ and $E^{00}$ in place of $E^0$.
\end{proof}

There is redundancy in Theorem \ref{T:closed}, inasmuch as a flat $A$ may be representable with different choices of $U$, $\pi$, $\zeta$, and $\theta$.  If $E{:}U$ has components that are sign balanced, and in \eqref{E:closedhbal} hyperbalanced, then $U$ can be made smaller and $\pi$ larger.  The smallest possible set $U$ is $U_\S(A)$.  Once $U$ is chosen, the most refined partition is $\pi(A{:}U^c)$, but sometimes $\pi$ may be chosen to combine components of $A{:}U^c$.  Then given $U$ and $\pi$, the switching function(s) $\zeta$ and $\theta$ are determined up to sign (for $\zeta$) and (for $\theta$) translation on each component of $A{:}U^c$ if $U=U_\S(A)$ and $\pi=\pib(A)$ but not necessarily with other choices.  This versatility in Theorem \ref{T:closed} is sometimes something to be careful about.

The flats of $\M(\Ups)$ form a geometric lattice, which we denote by $\Lat\Ups$.  
There is also the semilattice of hyperbalanced flats, which we call $\Latb\Ups$; it is a geometric semilattice that corresponds to the intersection semilattice of the hyperplane representation of $\Ups$, so it is of particular interest; see Section \ref{sec:hyperplanes}.

\sectionpage
\section{Matroid: coatoms and cocircuits}\label{sec:cocircuit}

The coatoms of $\Lat\Ups$, the lattice of closed sets, are of interest, but even more so their complements, the cocircuits of $\M(\Ups)$, and also the hyperbalanced coatoms, which are the maximal elements of $\Latb\Ups$ when $\Ups$ is not hyperbalanced.  
The coatoms and cocircuits are quite different depending on whether $\Ups$ is hyperbalanced or not.
On the whole it is easier to describe the cocircuits (which we may call hypercocircuits to distinguish them from the sign cocircuits in $\bfF(\S)$).  

We begin with a pair of lemmas.

\begin{lem}\label{L:sbasish}
Every basis of $\bfF(\S)$ is hyperbalanced.
\end{lem}

\begin{proof}
A basis cannot be hyperfrustrated because it does not contain any sign circuits.
\end{proof}

\begin{lem}\label{L:maxhbal}
Suppose $\Ups$ is hyperfrustrated.  Then every maximal hyperbalanced edge set $A$ is a coatom with $b_\S(A) = b_\S(E)$ and $\rk_\Ups A = \rk(\Ups) -1 = n-b_\S(E)$.
\end{lem}

\begin{proof}
A maximal hyperbalanced edge set $A$ is closed in $\M(\Ups)$, either by Theorem \ref{T:closure}, or simply because adding an edge to it makes it hyperfrustrated, which increases its rank.

By maximality $A$ is connected and spanning in each component of $\Ups$.  We use \eqref{E:closure} to prove that, for any edge $e \notin A$, $\clos_\Ups(A \cup e) = E$, by proving that any other edge $f \notin A\cup e$ is in $\clos_\Ups(A \cup e)$.  If $f \in E^0$, then $f \in \clos_\Ups(A\cup e)$.  Otherwise, $f$ is in a component $\Ups{:}X$.  If $(A \cup e){:}X$ is sign unbalanced, then $X \subseteq U_\S(A\cup e)$ so $f \in E{:}U_\S(A\cup e) \subseteq \clos_\Ups(A \cup e)$.  If $(A \cup e){:}X$ is sign balanced, we prove that $E{:}X$ is sign balanced.  By maximality of $A$, $(A\cup f){:}X$ is hyperfrustrated for $f \in (E\setm A){:}X$, so $f$ belongs to a non-neutral sign circuit in $A \cup f$, which can only be a positive circle because $A{:}X$ contains no negative figure.  Therefore, $f \in \bcl_\S(A)$.  Since $f$ was any edge in $(E \setm A){:}X$, it follows that $E{:}X$ is sign balanced.  Therefore, $X \in \pib(A \cup e)$ and $E{:}X \subseteq E{:}\pib(A\cup e) \subseteq \clos_\Ups(A \cup e)$.
\end{proof}

A \emph{hyperbalancing set} for $\Ups$ is an edge set whose deletion results in hyperbalance.

Note that, since $e_\infty$ can be treated as a non-neutral loose edge, the extended matroid $\M_\infty(\Ups)$ falls under the hyperfrustrated case, part (II) of the following theorem.

\begin{thm}\label{T:cocircuit}
Let $\Ups$ be a gain signed graph with signed graph $\S = (\G,\s)$.  The coatoms and cocircuits of $\M(\Ups)$ are of the following types.
\begin{enumerate}[{\rm I.}]
\item  If $\Ups$ is hyperbalanced, they are the coatoms and cocircuits of $\bfF(\S)$.

\item  If $\Ups$ is hyperfrustrated, it has both hyperbalanced and hyperfrustrated coatoms.  
  \begin{enumerate}[{\rm(A)}]
  \item For a hyperbalanced set $A$ the following properties are equivalent:

  \begin{enumerate}[{\rm(i)}]
  \item $A$ is a coatom.
  \item $A$ is a maximal hyperbalanced edge set.
  \item $A$ is the closure $\clos_\Ups(F)$ of a basis $F$ of $\bfF(\S)$.
  \item $A$ is the balance-closure $\bcl_\Ups(F)$ of a basis $F$ of $\bfF(\S)$.
  \end{enumerate}
  The corresponding hypercocircuits are the minimal hyperbalancing sets.

  \item  The hyperfrustrated coatoms are the coatoms of $\bfF(\S)$ that are hyperfrustrated.  
  The corresponding hypercocircuits are:
  \begin{enumerate}[{\rm(i)}]
  \item If $\G{:}X$ is a component such that $\Ups \setm X$ is hyperfrustrated (this includes the case where $\Ups \setm X$ contains a non-neutral loose edge), any cocircuit of $\bfF(\S{:}X)$.
  \item If $\G{:}X$ is a component such that $\Ups \setm X$ is hyperbalanced, 
	  \begin{enumerate}[{\rm(a)}]
	  \item a bond $D$ of $\G{:}X$ such that $[\Ups{:}X] \setm D$ is hyperfrustrated, if $\S{:}X$ is sign balanced;
	  \item either a minimal deletion set $D$ of $\S{:}X$ such that $[\Ups{:}X] \setm D$ is hyperfrustrated, or a cut $D$ of $\G{:}X$ of the types in Theorem \ref{L:sgmatroid}(D\ref{del})--(D\ref{bondu-del}) such that $[\Ups{:}X] \setm D$ is hyperfrustrated, if $\S{:}X$ is sign unbalanced.
	  \end{enumerate}
  \end{enumerate}
  
  \end{enumerate}
\end{enumerate}
\end{thm}

\begin{proof}
(I)  If $\Ups$ is hyperbalanced, its matroid $\M(\Ups)$ is the same as $\bfF(\S)$; see Theorem \ref{L:sgmatroid}.  

(II)  Now consider hyperfrustrated $\Ups$.

In part (A), it is clear that (i) implies (ii), and (ii) implies (i) by Lemma \ref{L:maxhbal}.

Parts (iii) and (iv) are equivalent by Theorem \ref{T:closure} and Lemma \ref{L:sbasish}. 

Assume (ii); then $\rk_\S(A) = \rk_\Ups(E)-1$ (by Lemma \ref{L:maxhbal}) $= \rk_\S(E)$ because $\Ups$ is hyperfrustrated.  Now we prove (iii) implies (i).  Since $F$ is hyperbalanced by Lemma \ref{L:sbasish}, $\rk_\Ups(F) = \rk_\S(F) = \rk_\S(E) = \rk_\Ups(E) - 1$, which implies that $A = \clos_\Ups(F)$ is a coatom.  By \eqref{E:closurehbal} and Lemma \ref{L:gbalclosure}, $A$ is hyperbalanced.

For part (B), suppose $A$ is a hyperfrustrated coatom of $\M(\Ups)$.  Then $\rk_\Ups A = \rk_\Ups E - 1$, which means $n-b_\S(A)+1 = n-b_\S(E)$, i.e., $b_\S(A) = b_\S(E) + 1$.  This means $\rk_\S A = \rk_\S E - 1$, so $A$ is contained in a coatom $A'$ of $\bfF(\S)$.  But $A'$ is therefore hyperfrustrated and its rank is $\rk_\Ups A' = \rk_\Ups E - 1$; by maximality of a coatom, $A=A'$.

In this case, a hypercocircuit is a sign cocircuit $D$; they are described in Theorem \ref{L:sgmatroid}.  Each $D$ is contained in a single component $\Ups{:}X$.  If $\Ups \setm X$ is hyperfrustrated, then any sign cocircuit is a hypercocircuit because deleting it does not eliminate hyperbalance.  If $\Ups \setm X$ is hyperbalanced, the sign cocircuit $D$ must be restricted to those whose deletion does not eliminate hyperbalance; hence by Theorem \ref{L:sgmatroid} we get the classification in (I)(B)(ii).
\end{proof}

The hyperbalanced coatoms of a hyperfrustrated $\Ups$, being the maximal elements of $\Latb\Ups$, are of prime importance for geometry, as we shall see in Section \ref{sec:hyperplanes}.

\sectionpage
\section{Matroid: independence and bases}\label{sec:indep}

Knowing the rank function of the matroid $\M(\Ups)$ we can determine the independent sets and bases.  They are complicated but they can be fully classified using the matroid invariant \emph{nullity}, defined as $\nul_\Ups S := |S| - \rk_\Ups S$ for a set $S$.  In particular, a set is independent if and only if its nullity is 0.
For a gain signed graph nullity has a nice expression in terms of the cyclomatic number $\xi$ and the quantity $u_\S(S)$, the number of sign-unbalanced components of $S$.  (Recall that this does not include loose edges; $u_\S(S)=0$ if $S$ has only loose edges.)

\begin{lem}\label{L:nullity}
The nullity of an edge set $S$ in $\M(\Ups)$ is 
\begin{align*}
\nul_\Ups S &= \begin{cases}
\xi(S) + b_\S(S) - c(S) 	&\text{if $S$ is hyperbalanced,} \\
\xi(S) + b_\S(S) - c(S) - 1	&\text{if $S$ is hyperfrustrated}
\end{cases} \\
&= \begin{cases}
\xi(S) - u_\S(S) 	&\text{if $S$ is hyperbalanced,} \\
\xi(S) - u_\S(S) - 1	&\text{if $S$ is hyperfrustrated}.
\end{cases}
\end{align*}
\end{lem}

\begin{proof}
The formula follows from the definition of nullity, the rank formula \eqref{E:rank}, and the definition of the cyclomatic number.
\end{proof}

The second version of the nullity formula shows that isolated vertices can be ignored in computing nullity.

We prepare for independence with a lemma.

\begin{lem}\label{L:2hfrust}
An edge set that contains two hyperfrustrated edge components is dependent.
\end{lem}

\begin{proof}
Say $S_1$ and $S_2$ are hyperfrustrated edge components of $S$.  Then 
$$
\nul_\Ups S = \xi(S) - u_\S(S) - 1 \geq [\xi(S_1) - u_\S(S_1)] + [\xi(S_2) - u_\S(S_2)] - 1
$$
since $\xi$ and $u_\S$ are additive on components.  
In order to be hyperfrustrated a component $S_i$ must contain a positive circle, a handcuff with two negative figures, or a loose edge.  In either case, $\xi(S_i) - u_\S(S_i) > 0$.  It follows that $\nul_\Ups S > 0$.
\end{proof}

A \emph{unicycle} is a circle or a half edge that may have trees attached.

\begin{thm}\label{T:indep}
The independent sets of $\M(\Ups)$ are the following types of edge set, considered as spanning subgraphs.
\begin{enumerate}[{\rm I.}]
\item Hyperbalanced:  Every component is a tree or a sign-unbalanced unicycle, and it has no loose edges.
\item Hyperfrustrated:  Every component is a tree or a sign-unbalanced unicycle, except that there may be one component $S_0$ that is either a sign-balanced unicycle, a loose edge, or a sign-unbalanced theta graph or a handcuff with at least one negative figure, possibly with attached trees, and in each case the unique sign circuit in $S_0$ is non-neutral.
\end{enumerate}
\end{thm}

\begin{figure}[htbp]
\includegraphics[scale=.6]{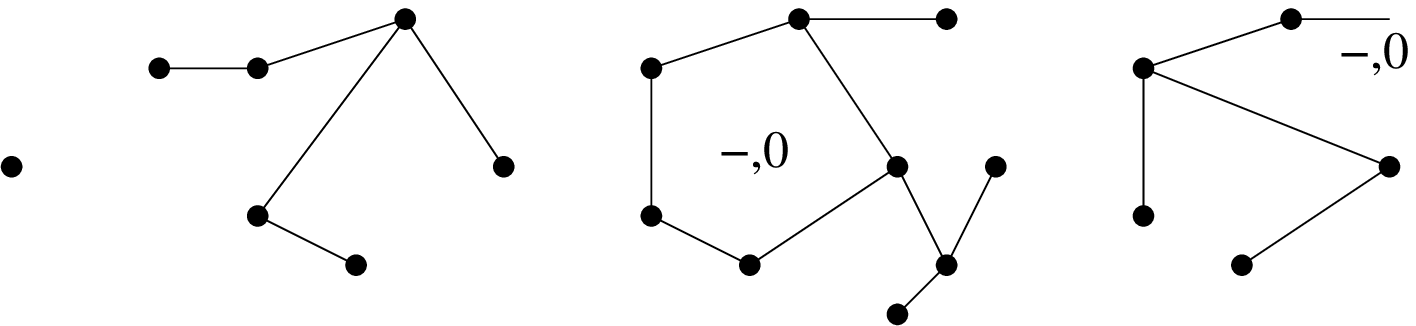}
\\[12pt]
\includegraphics[scale=.6]{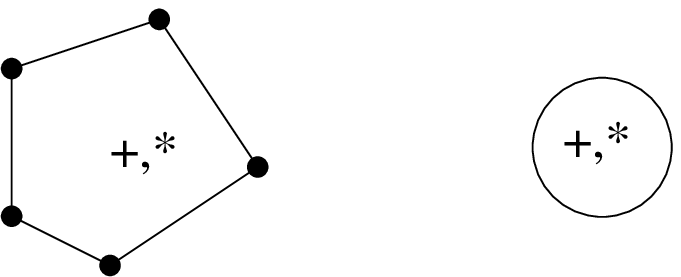}
\\[18pt]
\includegraphics[scale=.8]{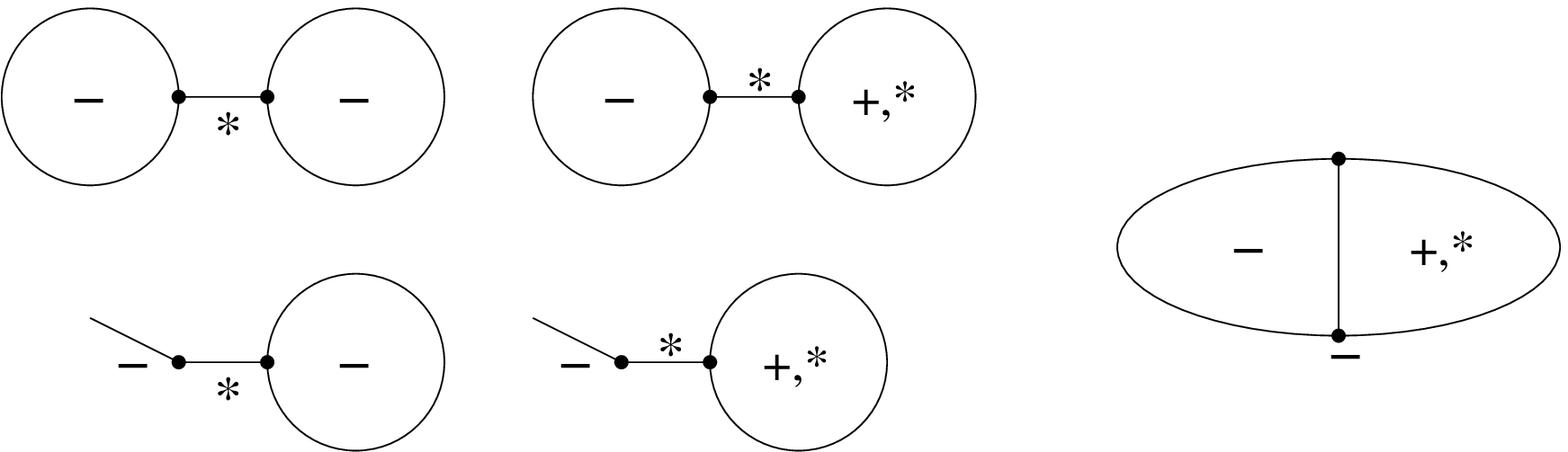}
\caption{Independent sets.  There may be any number of hyperbalanced components.  There may be any number of trees pendant from a component, not shown except in the top row.  Other than that, the figure shows all types of independent set.  Circle signs and gains ($*$ means the gain is not $0$) are indicated.  
Top row:  Possible hyperbalanced components: two trees and two neutral negative unicycles.  
Lower rows:  The possible hyperfrustrated component $S_0$.  Second row: sign balanced.  Bottom rows:  sign unbalanced.}
\label{F:indep}
\end{figure}

\begin{proof}
To find independent sets we assume $S$ is independent and examine how $\nul_\Ups S$ can equal 0.  
The equation we have to solve is $\xi(S) - u_\S(S) = 0$ or $1$ depending on whether $S$ is or is not hyperbalanced.  By independence, any positive circle or loose edge must be non-neutral.

In the hyperbalanced case $\xi(S) = u_\S(S)$ and $S$ must be independent in $\bfF(\Sigma)$.   Therefore, there can be no positive circles or loose edges in $S$.  Every sign-balanced component must be acyclic.  Each sign-unbalanced component must have no balanced circles and exactly one negative figure, so it is a unicycle.  This proves part (I).

In the hyperfrustrated case the components are the same except that, because $\xi(S) = u_\S(S) + 1$, there is one special component $S_0$ that contains a sign circuit, which must be non-neutral since $S$ is independent; thus $S_0$ is hyperfrustrated.  (By Lemma \ref{L:2hfrust} there cannot be two such components.)  
If $S_0$ is sign balanced, it has one circle, which is positive (and is not neutral), so it is a unicycle, or else it is a non-neutral loose edge.  

If $S_0$ is sign unbalanced, it has cyclomatic number 2 and at least one negative figure.  Thus it is either an unbalanced theta graph or an unbalanced handcuff, with possible pendant trees.  In either case it contains exactly one sign circuit, which must not be neutral.  Thus there is an edge $e \in S_0$ (in the sign circuit) such that $S \setm e$ is hyperbalanced.  That is, $S$ is obtained by adding $e$ to a hyperbalanced independent set $I$.
If $e$ is added to a tree component of $I$, $S_0$ is a unicycle.  In order to contain a sign circuit it must be a sign-balanced unicycle.  If $e$ is an isthmus in $S$, it joins two components of $I$, neither of which can be a tree because then $S_0$ would be hyperbalanced; thus the components must be sign-unbalanced unicycles and $S_0$ contains a handcuff sign circuit.  The third possibility is that $e$ is added to a unicyclic component of $I$ to form $S_0$ with $\xi(S_0) = 2$.  A component with cyclomatic number 2 must be, aside from any pendant trees, a theta graph or a handcuff.  Since it is sign unbalanced, the theta graph must have two negative and one positive circle and the handcuff must have one or two negative figures.  

In every case, $S_0$ contains a sign circuit that is the unique sign circuit in $S$, so it must be non-neutral to make $S$ hyperfrustrated.
\end{proof}

\begin{cor}\label{C:swrep}
Sign and gain switching and edge reorientation in $\Ups$ do not change the matroid $\M(\Ups)$ and therefore do not change linearly dependent subsets among the vectors $\z(E)$.
\end{cor}

\begin{proof}
The matroid is determined by its independent sets.  An independent set is characterized by its underlying graph, by the signs of its circles, and by the neutrality of its sign circuits.  None of these is affected by switching or orientation.  The only one of those assertions that is not obvious is the one about neutrality, but Lemma \ref{L:swgainswalk} says the gain of a sign circuit is not changed by gain switching.

The dependent vector sets are determined by $\M(\Ups)$, according to Theorem \ref{T:rank}.
\end{proof}

We can characterize the maximal independent sets.

\begin{thm}\label{T:bases}
The bases of $\M(\Ups)$ are the following types of edge set $B$, considered as spanning subgraphs.
\begin{enumerate}[{\rm I.}]
\item If\/ $\Ups$ is hyperbalanced, then in a sign-balanced component $\Ups_i$ of $\Ups$, $B \cap E(\Ups_i)$ is a spanning tree; and in a sign-unbalanced component $\Ups_i$, $B \cap E(\Ups_i)$ is a spanning disjoint union of sign-unbalanced unicycles.  $B$ contains no loose edges.
\item If\/ $\Ups$ is hyperfrustrated, $B$ is the same except that either $B$ has one edge component that is a non-neutral loose edge; or in one sign-balanced component $\Ups_i$ that is not a tree, $B \cap E(\Ups_i)$ is a spanning sign-balanced unicycle whose circle is not neutral; or in one sign-unbalanced component $\Ups_i$, one component $B_0$ of $B \cap E(\Ups_i)$ is a theta graph or handcuff with at least one negative figure, possibly with attached trees, and the unique sign circuit in $B_0$ is non-neutral.
\end{enumerate}
\end{thm}

\begin{proof}
For an independent set $B$ to be maximal, in each sign-balanced component of $\Ups$ it must be a spanning tree or a spanning sign-balanced unicycle.  Also, in each sign-unbalanced component $\Ups_i$ of $\Ups$, $B$ cannot have any tree components so every component of $B \cap E(\Ups_i)$ is either a sign-unbalanced unicycle, or a graph of cyclomatic number 2 as described in Theorem \ref{T:indep}.  Furthermore, $B$ can have only one component that is not a tree or a sign-unbalanced unicycle, and none if $\Ups$ is hyperbalanced or if $B$ contains a non-neutral loose edge.
\end{proof}

\sectionpage
\section{Matroid: circuits}\label{sec:circuit}

Knowing the independent sets, it is finally time to find the circuits.  
A set of elements of a matroid is a circuit if and only if it has nullity 1 but every proper subset has nullity 0.
As usual, for $\M_\infty$ we can treat $e_\infty$ as if it were a non-neutral loose edge.

\begin{thm}[Hypercircuits]\label{T:hypercirc}
The hypercircuits are the following types of edge set.
\begin{enumerate}[{\rm I.}]
\item Hyperbalanced:  A neutral sign circuit, as in Figure \ref{F:hcirc-neutral}.
\item Hyperfrustrated; in every type, all sign circuits are to be non-neutral:
	\begin{enumerate}[{\rm a.}]
	\item Disconnected:  The union of two disjoint, non-neutral sign circuits, as in Figure \ref{F:hcirc-disconnected}.
	\item Sign balanced:  A theta graph or tight handcuff, as in Figure \ref{F:hcirc-balanced}.
	\item A subdivision of a sign-antibalanced $K_4$ in which all three subdivided quadrilaterals are non-neutral, as in Figure \ref{F:hcirc-fat-k4}(b).
	\item A quadruple path as in Figure \ref{F:hcirc-fat-k4}(a).
	\item A linked theta and circle as in Figure \ref{F:hcirc-thetacircle}.
	\item Three linked circles as in Figure \ref{F:hcirc-polycircles}.
	\end{enumerate}
\end{enumerate}
\end{thm}

\begin{figure}[hb]
\includegraphics[scale=.7]{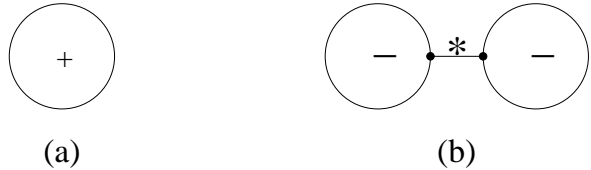}
\caption{The hyperbalanced hypercircuits.  
 In all these figures the asterisk $*$ denotes a path of length $\geq0$.  All other paths, including (closed paths), have length $\geq1$.  An isolated, positive circle may be a loose edge.  A negative circle with one indicated vertex may be a half edge.}
\label{F:hcirc-neutral}
\end{figure}

\begin{figure}[hb]
\includegraphics[scale=.7]{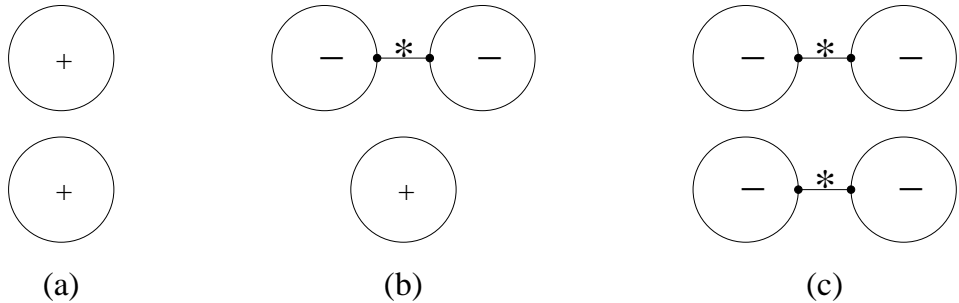}
\caption{From left to right, the three types of disconnected hypercircuit.  The sign circuits are non-neutral.}
\label{F:hcirc-disconnected}
\end{figure}

\begin{figure}[hb]
\includegraphics[scale=.7]{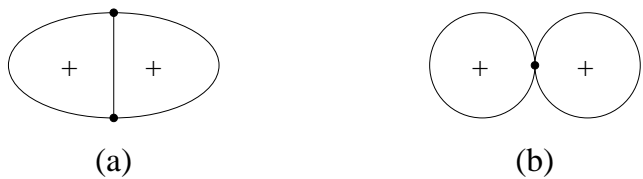}
\caption{The two types of sign-balanced hypercircuit.  All positive circles are non-neutral.}
\label{F:hcirc-balanced}
\end{figure}

\begin{figure}[hb]
\includegraphics[scale=.7]{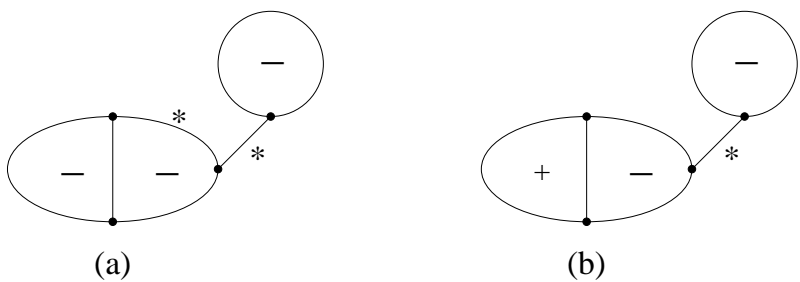}
\caption{The two types of linked theta-circle hypercircuit.  Every contained sign circuit is non-neutral.}
\label{F:hcirc-thetacircle}
\end{figure}

\begin{figure}[hb]
\includegraphics[scale=.7]{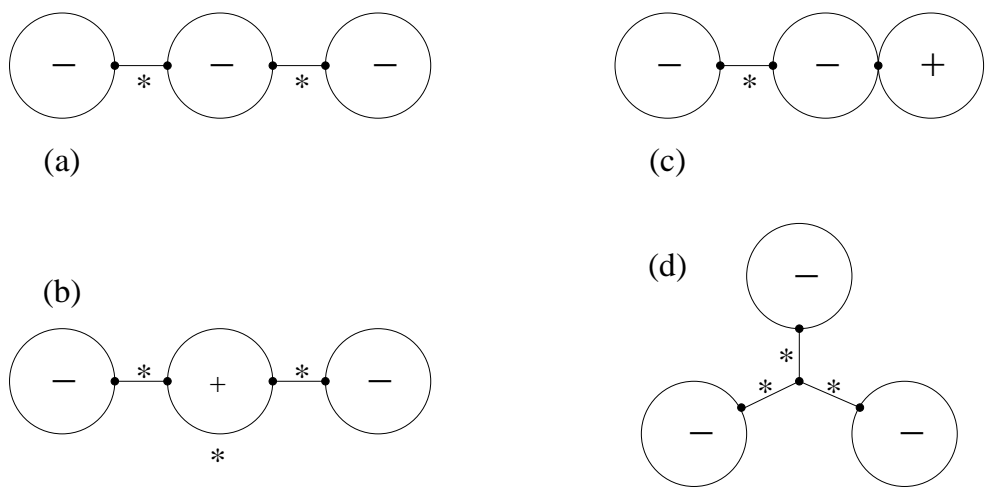}
\caption{The four kinds of linked-circle hypercircuit.  Each contained sign circuit is non-neutral.}
\label{F:hcirc-polycircles}
\end{figure}

\begin{figure}[hb]
\includegraphics[scale=.7]{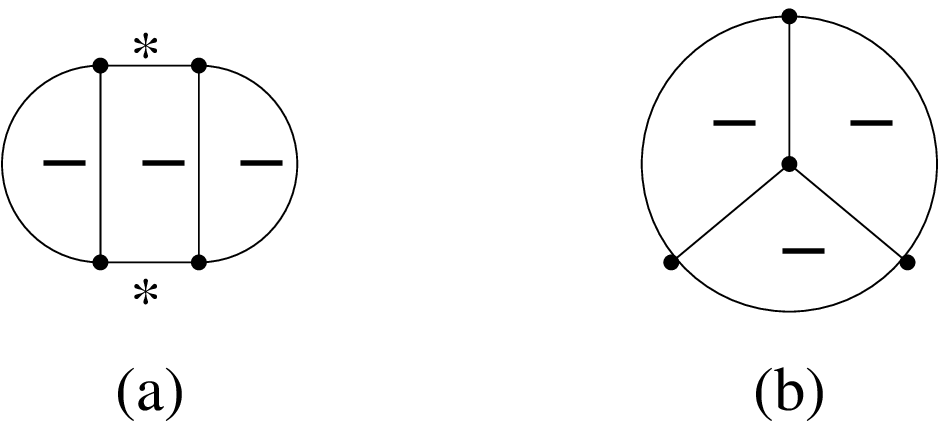}
\caption{Two special hypercircuit types: a quadruple path and a $K_4$ subdivision.  Again, each contained sign circuit is non-neutral.}
\label{F:hcirc-fat-k4}
\end{figure}

In the proof we use lollipops.  A \emph{lollipop} is a negative figure (a circle or half edge) with a path of length $\geq 0$ (the \emph{stick}) connected to it at one end.  The other end of the stick is the \emph{handle}.  (If the stick has length $0$, the handle is on the negative figure.)  A lollipop, if attached to another graph, is always attached at the handle.  It is convenient to call a lollipop \emph{positive} or \emph{negative} depending on the sign of the circle or half edge, irrespective of the stick signs.

\begin{proof}
The hyperbalanced hypercircuits $C$ are the neutral sign circuits because the matroid $\M(\Ups|C) = \bfF(\S|C)$.

For the rest of the proof we examine a hyperfrustrated hypercircuit $F$.  The nullity formula gives
\begin{equation}
\xi(F) = u_\S(F) +2.
\label{E:hfrust-hcct}
\end{equation}
Suppose $F$ has edge components $F_1, \ldots, F_k$.  Then 
$\xi(F) - u_\S(F) = \sum_i [\xi(F_i) - u_\S(F_i)] = 2.$
Since $\xi(F_i) \geq u_\S(F_i)$, this means $\xi(F_i) = u_\S(F_i)$ for all edge components but at most two.  Deleting those for which $\xi(F_i) = u_\S(F_i)$, we still satisfy \eqref{E:hfrust-hcct}, so by minimal dependency $F$ cannot contain such an edge component.  Therefore, $k\leq 2$.

If $F$ does have two edge components, each one satisfies $\xi(F_i) = u_\S(F_i) + 1$ so is hyperfrustrated.  Therefore, $F_i$ contains a non-neutral sign circuit $C_i$.  It follows from Lemma \ref{L:2hfrust} that $C_1 \cup C_2$ is already dependent; by minimality $F \subseteq C_1 \cup C_2$.  Deleting any edge gives an independent set, according to Theorem \ref{T:indep}, so equality holds and $F$ has type (IIa).

Now assume $k=1$, so $F$ is a single connected component.  Besides nullity 1 for $F$ itself we should have nullity 0 for every proper subset; that is,
$$
\xi(F\setm f) = u_\S(F \setm f) + \beta(f) \text{ for every } f \in F,
$$
where $\beta(f) = 0$ if $F\setm f$ is hyperbalanced and $1$ if it is not.  
In particular, $F$ must be formed by adding one edge, call it $e$, to an independent set $I$, and by minimal dependency it can have no pendant trees.  These rules allow us to construct candidates from the list in Theorem \ref{T:indep}.  

The nullity of an edge set is not changed by adding or subtracting a pendant edge.  A hypercircuit, therefore, has no tree components or pendant trees.   Suppose we add $e$ to $I$: no trees can remain pendant from $I \cup\{e\}$, so in particular any pendant trees in $I$ must combine with $e$ to form a path whose endpoints have degree at least 2.

Suppose $e$ joins two components of $I$, $I_1$ and $I_2$; neither can be a tree, so (from the structure theorem \ref{T:indep}) $I_1$ must be a negative lollipop.  We can get the same hypercircuit from a different $I$ in which the lollipop's negative figure is a circle formed from a path pendant from $I_2$.  This proves we may choose $e \in F$ so that $e$ is added to a single component of $I = F \setm \{e\}$, and as we assume $F$ is connected, so is $I$.

When we add $e$ to $I$, $e$ belongs to a set that is a maximal path whose interior vertices are divalent, or to a circle with one vertex attaching it to the rest of $F$, or to a circle that is $F$, or $e$ is a half edge with one vertex attaching it to the rest of $F$.  If $F$ is a circle, the circle must be positive and neutral since $F$ is dependent, so it falls under type (I).  In the other cases let us call the path or circle or half edge the \emph{$e$-ear}.  Let $A$ denote the $e$-ear and now let $I$ be the remainder of $F$.  $I$ is a connected independent set that is either a sign-unbalanced unicycle or one of the special types $S_0$ of Theorem \ref{T:indep}(II), and $\xi(F) = \xi(I) + 1$.

\emph{Case 0.  $F$ is a tight handcuff with two positive circles.}  This case has $\xi(F) = 2 = u_\S(F) + 2$, and deleting any edge gives nullity 1, so it is a hypercircuit.  This case is needed to rule out some examples.  In fact, we can now conclude that:

\begin{lem}\label{L:hcircposbal}
No hypercircuit can properly contain either a sign-balanced theta graph, or two positive circles with at most one common vertex.
\end{lem}

\emph{Case 1.  $I$ is a unicycle.}  Then $\xi(F) = 2$ so $F$ is either a handcuff or a theta graph, and \eqref{E:hfrust-hcct} implies that all circles are positive and there is no half edge in the handcuff.  The handcuff cannot contain a neutral circle, by minimality, so its two circles are positive and non-neutral; but then omitting the connecting edges (if any) gives a hypercircuit of type (IIa), so this gives type (IIa) or (IIb).  The theta graph similarly must have three non-neutral positive circles, which means that the paths $P_1, P_2, P_3$ between its two trivalent vertices must have distinct gains (computed with the same initial and final vertices).  Deleting any edge gives a positive, non-neutral circle with up to two pendant paths, which is independent, so this is a hypercircuit of type (IIb).

\emph{Case 2.  $I$ is a theta graph with a single positive circle.}  Here $F$ is dependent since $\xi(F) = 3$ and $\nul_\Ups F = \xi(F) - u_\S(F) - 1 = 1$.  The open question is whether every proper subset is independent.

Let $I$ have trivalent vertices $u,v$ and constituent paths $P_1, P_2, P_3$ with positive circle $P_1 \cup P_3$.  From Theorem \ref{T:indep}(II) it is sign-unbalanced and its positive circle is non-neutral.

The $e$-ear $A$ may be a circle or half edge; then we have a theta graph with a lollipop attached at its handle.  Lemma \ref{L:hcircposbal} implies that the lollipop is negative.  Because $I$ also contains a negative circle, deleting any one edge either reduces $\xi$ by 1 or, if the edge is in the stick, increases $u_\S$ by 1; in either case we get nullity 0.  Thus, any edge set of this type is a hypercircuit of the kind in Figure \ref{F:hcirc-thetacircle}.

The alternative is that $A$ is a path.  

Suppose $A$ has endpoints in the interiors of different constituent paths.  This gives a homeomorph of $K_4$.  To analyze the signs and gains we may consider it to be $K_4$.  A signed $K_4$ has an even number of positive triangles, so $I$ has either no or two such triangles.  If it has none, we have a sign-antibalanced $K_4$ as in type (IIc).  If it has two, it contains a sign-balanced theta graph which is already a hypercircuit of type (IIb).

Otherwise, the endpoints of $A$ are contained within one constituent path $P_i$.  Here we have two subcases.  If $A$ has endpoints $u$ and $v$, then $F$ consists of four internally disjoint $uv$-paths and there must be two paths of each sign to avoid having a balanced theta graph as a proper subgraph, a violation of Lemma \ref{L:hcircposbal}.  This gives the unique $F$ that is a quadruple path.  If the endpoints are not $u$ and $v$, $A$ forms a circle $C$ with all or part of $P_i$.  This circle cannot be positive, because if it were, $F$ would properly contain a configuration forbidden by Lemma \ref{L:hcircposbal}.  Thus, $C$ is negative.  Then $F$ is the graph in Figure \ref{F:hcirc-fat-k4}(a).

\emph{Case 3.  $I$ is a handcuff and $F$ does not contain a theta graph.}  (If $F$ contains a theta graph $\Theta$, it falls under type (IIb) if $\Theta$ is sign balanced and by suitable choice of $e$ it has already been treated under Case 2 if $\Theta$ is sign unbalanced.)  
The handcuff has one or two negative figures.  

A negative lollipop can be attached anywhere; then $\xi(F) = 3 = u_\S(F) + 2$ so this set is dependent.  
A positive lollipop can only be attached if $I$ has two negative figures (by Lemma \ref{L:hcircposbal}),
but as this duplicated adding a negative lollipop to a handcuff with one positive circle, we need not consider it separately.
Figure \ref{F:hcirc-polycircles} shows the possibilities for $F$ and makes it clear that deleting any edge $f$ makes $F\setm f$ independent.  

The remaining possibility for $F$ is that it is $I$ with a path ear $A$.  As $F$ cannot contain a theta graph, $A$ must have endpoints in the connecting path of $I$, which means $F$ has three circles or half edges, of which no two have more than one common vertex and no two can be positive.  This case is like adding a lollipop to a handcuff as was just treated.

That completes the analysis of possible hypercircuits.
\end{proof}

The many topological types of hypercircuit make a complicated list so we classify them in another way.  All sign circuits are non-neutral except in type (1).  
\emph{Sign contrabalanced} means without positive circles.

\begin{cor}\label{C:hypercircuits}
These are the hypercircuits:

\begin{enumerate}[{\rm(1)}]
\item A sign circuit that is neutral.

\item A disconnected hypercircuit is the union of two disjoint non-neutral sign circuits.

\item A connected hypercircuit that is not a neutral sign circuit and contains at most two positive circles.
  \begin{enumerate}[{\rm(A)}]
  \item If it contains no positive circle, it is either three negative lollipops joined at their handles, or a negative circle with two negative lollipops attached at two different vertices.
  \item If it contains exactly one positive circle, it is either:
    \begin{enumerate}[{\rm(i)}]
    \item A sign-contrabalanced handcuff with a positive circle attached at one vertex, or a positive circle with two negative lollipops attached at two different vertices. 
    \item A sign-unbalanced theta graph with one negative lollipop attached.
    \end{enumerate}
  \item If it contains two positive circles, it is one of:
    \begin{enumerate}[{\rm(i)}]
    \item Two positive circles intersecting at exactly one vertex.
    \item A sign-unbalanced theta graph with a path ear $Q$ added on a theta path $P$, with $Q$ signed so the unique circle in $P \cup Q$ is negative.
    \end{enumerate}
  \end{enumerate}

\item A sign-balanced theta graph in which every circle is non-neutral.

\item A subdivision of a sign-antibalanced $K_4$.
\end{enumerate}
\end{cor}

\sectionpage
\section{Minors}\label{sec:minors}

A \emph{minor} of $\Ups$ is the result of deleting and contracting edges.  The important property of minors of a gain signed graph is that they coordinate with minors of the associated matroid.

Deletion of an edge set $S$, denoted by $\Ups \setm S$, is obvious.  Contracting an edge set involves contraction in a signed graph, so we first define that.

\begin{defn}[Contraction of a Signed Graph \cite{SG}]\label{D:scontract}
Contracting a signed graph $\S$ by an edge set $S$ gives a signed graph $\S/S$.  
Its vertex set is $\pib(S)$.  Its edge set is $E \setm S$.  

For  a vertex $v \in V(\S)$, we denote by $B_v$ the block of the partial partition $\pib(S)$ that contains $v$, if there is one, that is, if $v$ is in a balanced component of $S$.  If $v$ is in an unbalanced component, then no $B_v$ exists.

An edge $e \in E \setm S$ becomes an edge in $\S/S$ with endpoints determined as follows.  
First, switch $\S$ so that in every balanced component all edges are positive.  

If $e=e_{vw}$ is a link or loop in $\S$, its endpoints in $\S/S$ are those of the sets $B_v, B_w \in \pib(S)$ that exist; then the sign of $e$ in $\S/S$ is its sign in the switched graph $\S$.  If one or both do not exist, $e$ has one or no endpoints in the contraction, thus becoming a half edge or loose edge.

If $e=e_v$ is a half edge in $\S$, it is a half edge in $\S/S$ with endpoint $B_v \in \pib(S)$ if $B_v$ exists, but it is a loose edge if $B_v$ does not exist.

If $e$ is a loose edge in $\S$, it is a loose edge in $\S/S$.

The contraction $\S/S$ is well defined up to switching of signs.
\end{defn}

\begin{lem}\label{L:sbalcontract}
In $\S$, suppose $S \subseteq E$ and $T \subseteq E \setm S$.

If $S$ is balanced in $\S$, then $T$ is balanced in $\S/S$ if and only if $S \cup T$ is balanced in $\S$.

If $S$ is unbalanced in $\S$ and $S \cup T$ is connected, then $T$ is unbalanced in $\S/S$ if $V(S \cup T) \supset U_\S(S)$, while $T$ consists of loose edges in $\S/S$ if $V(S \cup T) = U_\S(S)$.

For any edge set $S$, $b_\S(S \cup T) = b_{\S/S}(T)$.
\end{lem}

\begin{proof}
The first part is \cite[Lemma 4.1]{SG}.  The formula $b_{\S/S}(T) = b_\S(S \cup T)$ results from applying this to each component of $S \cup T$.

For the second part, suppose $V(S \cup T) \supset U_\S(S)$.  Then there is a link $e$ of $T$ that is not in $S$ and is incident to exactly one vertex of $U_\S(S)$.  In the contraction $e$ becomes a half edge; therefore $T$ is unbalanced in $\S/S$.  If however $V(S \cup T) = U_\S(S)$, then every edge of $T$ becomes a loose edge.

For the third part, consider each component $A_i$ of $A = S \cup T$ separately.  If $S \cap A_i$ is balanced in $\S$, then $b_{\S/S}(T \cap A_i) = b_\S(A_i)$ by the first part.  If $S \cap A_i$ is unbalanced, then $b_{\S/S}(T \cap A_i) = 0$ by the second part.  The general formula follows by addition over the $A_i$.
\end{proof}

\begin{defn}[Contraction of a Gain Signed Graph]\label{D:gscontract}
The contraction of an edge set $S$ in $\Ups$ is a gain signed graph or a signed graph, in either case denoted by $\Ups/S$.
These are the two cases.
\begin{enumerate}[1.]
\item  If $S$ is hyperbalanced, switch the gains on $\Ups$ so every edge in $S$ is neutral; then apply signed-graph contraction by $S$.  All edges of $E \setm S$ retain their (switched) gains.  This contraction is a gain signed graph.
\item  If $S$ is hyperfrustrated, erase all gains, leaving only the signed graph $\S$, and contract $S$ in $\S$.  This contraction is a signed graph without gains (or, it can be viewed as a gain signed graph in which all edges are neutral).
\end{enumerate}
The contraction $\Ups/S$ is well defined up to switching of signs and, in the first case, of gains.
\end{defn}

The first theorem says that minors are independent of the order of operations.  In particular, a minor can be computed by deleting and contracting one edge at a time.

\begin{thm}\label{T:minorsdef}
Suppose $S, T$ are disjoint subsets of $E$.  Then $(\Ups \setm S) \setm T = \Ups \setm(S \cup T)$, $(\Ups/S) \setm T = (\Ups \setm T)/S$, and $(\Ups/S)/T = \Ups/(S \cup T) = (\Ups/T)/S$.
\end{thm}

The same theorem for signed graphs is \cite[Proposition 4.2]{SG}.  As in \cite{SG}, equality in the contraction formulas has to be interpreted as allowing for certain name changes, as the vertex set of $\Ups/(S\cup T)$ is $\pib(S\cup T)$ in $\S$, while that of $(\Ups/S)/T$ is $\pib(T)$ in $\S/S$.  We handle this by identifying vertex sets in the natural way.  We use the notation of signed graphs, since the gains do not affect the vertex sets of contractions.  For $v \in V(\S)$, let $B_v(S)$ be the set in $\pib(\S,S)$, defined in $\S$, that contains $v$; similarly, for $B_v \in V(\S/S)$ let $B'_{B_v}(T)$ be the set in $\pib(\S/S,T)$, defined in $\S/S$, that contains $B_v(S)$, and finally, let $B_v(S\cup T)$ be the set in $\pib(\S,S \cup T)$ that contains $v$.  Then it is a fact that $B_v(S \cup T) = \bigcup B'_{B_v}(T)$, the union of all sets $B_w(S) \in B'_{B_v}(T)$, so we can identify the vertex $B_v(S \cup T)$ in $\S/(S\cup T)$ with the vertex $B'_{B_v}(T)$ in $(\S/S)/T$.

We need a gain analog of Lemma \ref{L:sbalcontract}.

\begin{lem}\label{L:hbalcontract}
Suppose $S \subseteq E$ is hyperbalanced in $\Ups$ and $T \subseteq E \setm S$.  Then $T$ is hyperbalanced in $\Ups/S$ if and only if $S \cup T$ is hyperbalanced in $\Ups$.
\end{lem}

\begin{proof}
We depend on the fact that contraction of $\Ups$ is built upon contraction in $\S$.

If $S \cup T$ is hyperbalanced, switch so its edges are all neutral.  Contracting $S$ leaves the gains on $T$ neutral, so $T$ is hyperbalanced in $\Ups/S$.

Conversely, suppose $T$ is hyperbalanced in $\Ups/S$.  Assume the gains in $\Ups$ have been switched so every edge of $S$ is neutral.  Now, switch $\Ups/S$ so every edge in $T$ is neutral in $\Ups/S$.  Switching gains in $\Ups/S$ by $\theta$, to $\phi^\theta$, can be applied to $\Ups$: define $\theta' : V \to \K^+$ by $\theta'(v) = \theta(B_v)$, since $B_v$ is defined for every vertex $v \in V$, due to balance of $S$.  Switching $\Ups$ in this way neutralizes every edge in $S \cup T$, proving hyperbalance of $S \cup T$, and gives the contraction the switched gains $\phi^\theta$.
\end{proof}

\begin{proof}[Proof of Theorem \ref{T:minorsdef}]
The parts with deletions are routine.  Signed graph minors obey the formulas in Theorem \ref{T:minorsdef} \cite[Proposition 4.2]{SG}.  Contraction of a gain signed graph follows the rules for signed-graph contraction supplemented by a rule for gains, so the only question is how the gains behave.  If $S\cup T$ is hyperbalanced, we may assume by gain switching that $\phi|_{S\cup T} \equiv 0$; then the gains off $S \cup T$ are never changed in any of the contractions.  If $S \cup T$ is not hyperbalanced, $\Ups/(S \cup T)$ has no gains.  In $\Ups/S$, if $S$ is hyperfrustrated there are no gains, so there are no gains in $(\Ups/S)/T$; while if $S$ is hyperbalanced, Lemma \ref{L:hbalcontract} tells us that $T$ is hyperfrustrated in $\Ups/S$ so there are no gains in $(\Ups/S)/T$.  Either way, there are no gains in any of $(\Ups/S)/T$, $\Ups/(S \cup T)$, and $(\Ups/T)/S$, so they are equal by \cite[Proposition 4.2]{SG}.
\end{proof}

\begin{thm}\label{T:minorsmatroid}
Suppose $S \subseteq E$.  Then $\M(\Ups \setm S) = \M(\Ups) \setm S$ and $\M(\Ups/S) = \M(\Ups)/S$.
\end{thm}

\begin{proof}
We prove the rank functions agree.  This is trivial for deletion.  

For contraction we compute the ranks in $\M$ and $\M/S$.  With $A \subseteq E \setm S$, first from the definition of rank in a contraction matroid:
\begin{align*}
\rk_{\M(\Ups)/S} A &= \rk_\M(A \cup S) - \rk_\M(S) 
\\&= [n - b_\S(A \cup S) + \delta_\Ups(A\cup S)] - [n - b_\S(S) + \delta_\Ups(S)] 
\\&= b_\S(S) - b_\S(A \cup S) + \delta_\Ups(A\cup S) - \delta_\Ups(S),
\end{align*}
and second from rank in the contracted gain signed graph:
\begin{align*}
\rk_{\M(\Ups/S)} A &= |V(\Ups/S)| - b_{\S/S}(A) + \delta_{\Ups/S}(A) 
\\&= b_\S(S) - b_{\S/S}(A) + \delta_{\Ups/S}(A) .
\end{align*}
Now we consider two cases.

If $S$ is hyperbalanced, then $\delta_{\Ups/S}(A) = \delta_\Ups(A \cup S)$ by Lemma \ref{L:hbalcontract}, so
$$
\rk_{\M(\Ups)/S} A = b_\S(S) - b_\S(A \cup S) + \delta_\Ups(A\cup S)
$$
and
$$
\rk_{\M(\Ups/S)} A = b_\S(S) - b_{\S/S}(A) + \delta_\Ups(A\cup S).
$$
By Lemma \ref{L:sbalcontract} $b_\S(A \cup S) = b_{\S/S}(A)$.  Therefore, the ranks are equal.

If $S$ is hyperfrustrated, then $\Ups/S = \S/S$ so $\M(\Ups/S) = \bfF(\S/S)$.  The contracted matroid computation is
\begin{align*}
\rk_{\M(\S)/S} A &= \rk_\M(A \cup S) - \rk_\M(S) 
\\&= [n - b_\S(A \cup S)] - [n - b_\S(S)] 
\\&= b_\S(S) - b_\S(A \cup S).
\intertext{The computation in the contracted graph is}
\rk_{\M(\Ups/S)} A &= |V(\S/S)| - b_{\S/S}(A)
\\&= b_\S(S) - b_{\S/S}(A).
\\&= b_\S(S) - b_\S(A \cup S)
\end{align*}
by Lemma \ref{L:sbalcontract}.  Thus, the ranks are equal in this case as well.
\end{proof}

Contraction of a single edge is sufficiently important to merit separate statement.

\begin{cor}\label{C:edgecontraction}
For $e \in E$, the contraction $\Ups/e$ is described in the following list:
\begin{enumerate}[{\rm(a)}]
\item\label{C:edgecontraction:link} If $e$ is a link $e_{uv}$, switch signs and gains so it is positive and neutral, identify $u$ and $v$, and delete $e$.  All other edges retain their switched signs and gains.
\item\label{C:edgecontraction:loose} If $e$ is a positive loop or loose edge and is neutral, delete it and retain all signs.  If $e$ is neutral, retain all gains.  If not, erase all gains.
\item\label{C:edgecontraction:half} If $e$ is a half edge or negative loop at vertex $v$, delete it and $v$, thus removing $v$ as an endpoint from any other edge incident with $v$.  Retain the signs of edges that do not become loose or half edges; make loose edges positive and half edges negative.  If $e$ is neutral, retain all gains.  If not, erase all gains.
\item\label{C:edgecontraction:extra} If $e$ is the extra point $e_\infty$, delete it and erase all gains but retain signs.
\end{enumerate}

In the extended matroid $\M_\infty(\Ups)$, when a neutral edge or a link is contracted the extra point $e_\infty$ remains the extra point.  When a non-neutral, non-link edge is contracted, $e_\infty$ becomes a matroid loop; it may be treated as a loose edge in the resulting signed graph.
\end{cor}

\sectionpage
\section{The hyperplane model}\label{sec:hyperplanes}

The biggest reason for gain signed graphs is that they encode popular hyperplane arrangements in a more manageable form.

\subsection{Affine hyperplanes}\label{sec:affinehyp}\

The hyperplane associated to an edge $e_{ij}$ lies in the affine space $\bbA^n(\K)$ and is given by 
$$\h(e):  \tau(v_i,e)x_i + \tau(v_j,e)x_j = -\phi(e).$$  
(The affine space is distinguished from the vector space $\K^n$ by having inhomogeneous as well as homogeneous subspaces.)  
The hyperplane of a half edge $e$ at $v_i$ is $x_i = -\phi(e)$.  The ``honorary hyperplane'' of a loose edge is given by the equation $0=-\phi(e)$, which is the \emph{degenerate hyperplane} $\h(e) = \bbA^n(\K)$ if $\phi(e)=0$ and the \emph{phantom hyperplane} $\eset$ if $\phi(e) \neq 0$.  (That makes sense in the projectivization.)  The honorary hyperplane of the extra point $e_\infty$ is the phantom hyperplane.
Thus from a gain signed graph $\Ups$ with gain group $\K^+$ we get the hyperplane arrangement
$$
\cA[\Ups] := \{ \h(e):  e \in E \}.
$$
The intersection semilattice of this arrangement is
$$
\cL(\cA[\Ups]) := \{ \textstyle\bigcap\cS : \cS \subseteq \cA[\Ups],\ \bigcap\cS \neq \eset \}.
$$

We list some arrangements of the type gain signed graphs are intended for, which we call \emph{(signed) affinographic} since they are affine deformations of graphic or signed-graphic hyperplane arrangements.
By $\eps e_{ij}$ we denote an edge with endpoints $v_i$ and $v_j$ and sign $\eps$, while by $\pm e_{ij}$ we denote two edges with both signs, a positive edge $+e_{ij}$ and a negative edge $-e_{ij}$.  A simple graph $\G$ thus gives rise to a signed graph $\eps\G$, in which all edges have the same sign $\eps$, and to $\pm\G = (+\G) \cup (-\G)$, in which all edges are doubled with both signs.
\begin{enumerate}[({Af}1)]
\item\label{Adef} Affine arrangements whose edge sets have forms like $\{ (+e_{ij},g) : i< j,\ -k \leq g \leq l \}$, known as deformations of the (all-positive) complete-graphic arrangement $\cA_{n-1}$ with edge set $\{ (+e_{ij},0) : i< j \}$.
\item\label{Bdef} Affine arrangements whose edge sets have forms like $\{ (\pm e_{ij},g) : i< j,\ -k \leq g \leq l \} \cup \{ (-e_{ii},g),\ -k \leq g \leq l \}$, known as deformations of the complete signed-graphic arrangement $\cB_n$ with edge set $\{ (\pm e_{ij},0) : i< j \} \cup \{ (-e_{ii},0) \}$.
\item\label{Ddef} Affine deformations whose edge sets have forms like $\{ (\pm e_{ij},g) : i<j,\ -k \leq g \leq l \}$, of the signed-graphic arrangement $\cD_n$ with edge set $\{ (\pm e_{ij},0) : i< j \}$.
\item\label{Cdef} An affinographic arrangement of Coxeter type, whose edge set is as in (Af\ref{Bdef}) or (Af\ref{Ddef}) with $k=0$ \cite[Theorem 3.5]{Athan}.
\item\label{Spm} The sign-symmetric Shi arrangement, whose edge set is $\{ (\pm e_{ij},g) : i<j,\ g=0,1\}$, and a variety of similar arrangements \cite[Section 3]{Athan}.
\item\label{St} The Shi threshold arrangement, whose edge set is $\{ (-e_{ij},g) : i<j,\ g=0,1\}$ \cite{Seo1}, and the similar arrangement in \cite[Theorem 5.4]{Athan}.
\item\label{Lt} The Linial threshold arrangement with coordinate and shifted hyperplanes, whose edge set is $\{ (-e_{ij},1) : i<j \} \cup \{ (-e_i,g) : g = 0,1 \}$ \cite{Song1, Song2, Song3, Song4}.
\item\label{Ct} The Catalan threshold arrangement, whose edge set is $\{ (-e_{ij},g): i<j,\ g=0,\pm1 \}$ \cite{Seo2}.
\item\label{gent} A generalized threshold arrangement, whose edge set is $\{ (-e_{ij},g): i<j,\ -k \leq g \leq l \}$ \cite{Bala}.
\end{enumerate}

\subsection{Projective hyperplanes}\label{sec:projhyp}\

The link between this affine arrangement and the vector model in linear space is through the \emph{projectivization}, 
$$\cA_\bbP[\Ups] := \{ \h(e)_\bbP: e \in E \} \cup \{h_\infty\},$$
where $h_\infty$ denotes the ideal or infinite hyperplane while $\h(e)_\bbP$ is the extension of $\h(e)$ into the projective space $\bbP^n(\K)$ if $\h(e)$ is not the phantom hyperplane and $\h(e)_\bbP = h_\infty$ if $\h(e)$ is the phantom hyperplane.  (The affine part of $h_\infty$ is $\eset$, which should explain the name ``phantom hyperplane''.)
We define $\h_\bbP(e) := \h(e)_\bbP$ in order to have the function $h_\bbP$ that maps edges to projective hyperplanes.
The addition of the ideal hyperplane in the projectivization ensures that, in the real case, the regions remain the same, and in the complex case, the complement of the arrangement remains the same.  

The intersection lattice of the projective arrangement is 
$$\cL(\cA_\bbP[\Ups]) := \{ \textstyle\bigcap\cS : \cS \subseteq \cA_\bbP[\Ups] \}.$$
This is the lattice of closed sets of a matroid $\M(\cA_\bbP)$ whose ground set is the set of hyperplanes and whose rank function is $\rk\cS = \codim\bigcap\cS$.
The affine intersection semilattice $\cL(\cA[\Ups])$, which is a meet subsemilattice of $\cL(\cA_\bbP[\Ups])$, is a \emph{geometric semilattice}, as defined by Wachs and Walker \cite{WW}.  One definition is that it consists of the flats of a geometric lattice that do not lie above a fixed atom; in our case the lattice is $\cL(\cA_\bbP[\Ups])$ and the atom is $h_\infty$.  
The corresponding matroid-like structure is a \emph{semimatroid}, defined subsequently by Ardila \cite{Ard}.

The coordinates in projective space are homogeneous coordinates $[x_0,x_1,\ldots,x_n]$ (not all zero).  Projecting the vector model in $\K^{1+n}$ to homogeneous coordinates in $\bbP^n(\K)$, the hyperplane $\h_\bbP(e)$ is the dual space of the projected vector $[\z(e)]$; that is, $\h_\bbP(e) = \{ [\bfy] \in \bbP^n(\K):  [\z(e)]\cdot[\bfy]=0\}$.

\begin{thm}[Hyperplane Representations]\label{T:hyperplanes}
The matroid $\M(\cA_\bbP[\Ups])$ is isomorphic to $\M_\infty(\Ups)$ by the mapping $\h_\bbP$, which induces a lattice isomorphism $\cL(\cA_\bbP[\Ups]) \cong \Lat\M_\infty(\Ups)$ and a semilattice isomorphism $\cL(\cA[\Ups]) \cong \Latb\Ups$.
\end{thm}

\begin{proof}
This follows from Theorem \ref{T:rank} by vector-space duality.
\end{proof}

\subsection{Regions and polynomials}\

Consider a hyperplane arrangement $\cA$ in $\bbA^d(\K)$.  
The \emph{characteristic polynomial} of $\cA$ is 
\begin{align*}
p(\cA;\lambda) &:= \sum_{{\cS \subseteq \cA:}\,{\bigcap\cS \neq \eset}} (-1)^{|\cS|} \lambda^{\dim\bigcap\cS} 
= \sum_{s \in \cL(\cA)} \mu(\hat0,s) \lambda^{\dim s} ,
\end{align*}
where $\hat0 = \bbA^d(\K)$, the bottom element of $\cL(\cA)$, and $\mu$ is the M\"obius function.

We define the \emph{chromatic polynomial} of $\Ups$:
\begin{align*}
\chi_\Ups(\lambda) &:= \sum_{S \subseteq E:} (-1)^{|S|} \lambda^{n-\rk_\Ups S} 
= \sum_{S \in \Lat\Ups} \mu(\eset,S) \lambda^{n-\rk_\Ups S} ,
\end{align*}
which is interpreted as $0$ if $\Ups$ has any neutral loose edges.  A variant is
$\chi_{\Ups_\infty}(\lambda)$, which is $\chi_{\Ups\cup\{e_\infty\}}(\lambda)$ where $e_\infty$ is interpreted as a non-neutral loose edge.
The \emph{balanced chromatic polynomial} of $\Ups$ is
\begin{align*}
\chib_\Ups(\lambda) &:= \sum_{{S \subseteq E:}\,\text{hyperbalanced}} (-1)^{|S|} \lambda^{n-\rk_\Ups S} 
= \sum_{S \in \Latb\Ups} \mu(\eset,S) \lambda^{n-\rk_\Ups S} ,
\end{align*}
which also is $0$ if $\Ups$ has neutral loose edges.  
(Note that a graph has a chromatic polynomial, while an arrangement (or matroid) has a characteristic polynomial.  One expects a chromatic polynomial to count colorings, as is the case with signed graphs and gain graphs; we hope to present such an interpretation separately.)

\begin{lem}\label{L:polys}
Assuming that $\Ups$ has no neutral loose edges, we have
$p(\cA[\Ups];\lambda) = \chib_\Ups(\lambda)$
and 
$p(\cA_\infty[\Ups];\lambda) = \chi_{\Ups_\infty}(\lambda)$.
\end{lem}

\begin{proof}
This is a consequence of Theorem \ref{T:hyperplanes} and the semilattice expressions for the polynomials.
\end{proof}

We are finally ready to count the regions of a real hyperplane arrangement described by a gain signed graph.

\begin{thm} \label{T:regions}
Consider the arrangements $\cA[\Ups]$ in $\bbA^n(\bbR)$ and $\cA_\infty[\Ups]$ in $\bbR^{1+n}$.

The number of regions of $\cA[\Ups]$ equals $(-1)^n \chib_\Ups(-1)$.  The number of bounded regions equals $(-1)^n \chib_\Ups(1)$.

The number of regions of $\cA_\infty[\Ups]$ equals $(-1)^{n+1} \chi_{\Ups_\infty}(-1)$.
\end{thm}

\begin{proof}
Special cases of \cite[Theorem A]{FUTA}, which applies to all real affine hyperplane arrangements, are that $\cA_\infty[\Ups]$ has $(-1)^{n+1} p(\cA_\infty[\Ups];-1)$ regions and $\cA[\Ups]$ has $(-1)^n p(\cA[\Ups];-1)$ regions.  
A special case of \cite[Theorem C]{FUTA} is that $\cA[\Ups]$ has $(-1)^n p(\cA[\Ups];1)$ bounded regions.  Applying Lemma \ref{L:polys} gives the theorem.
\end{proof}

\sectionpage
\section{Abstract abelian gains}\label{sec:abstract}

The theory we developed for an additive group of a field largely applies to any abelian gain group, except of course for the vector representation.  We demonstrate that here by presenting a purely combinatorial proof that $\M(\Ups)$ and $\M_\infty(\Ups)$ are matroids.  All the analysis in the preceding sections applies without change except that $\fG^n$ may not be a vector space.

We had hoped to treat general groups, not necessarily abelian, in which the additive definitions become multiplicative and multipliers become exponents, but the failure of commutativity led to the failure of so many properties from Section \ref{sec:tech} that we were unable even to establish that neutrality of sign circuits was independent of the method of calculating the gain.  Thus, we leave open the problem of dealing with nonabelian gain groups.

An \emph{abelian gain signed graph} is a triple $\Ups = (\G,\s, \phi)$ where $\G=(V,E)$ is a graph, the \emph{signature} or sign function $\s$ gives each edge an element of the sign set $\{+1,-1\}$, and the \emph{gain function} $\phi: \vec{E} \to \fG$ is an oriented labelling of edges from an abelian group $\fG$, which we write additively so that, for instance, $\phi(e\inv) = -\phi(e)$.  All the formulas and results of Section \ref{sec:tech} remain valid if we replace the group $\K^+$ by $\fG$.

We define rank exactly as in Equation \eqref{E:rank}.  
The task is to prove this is a matroid rank function without using a representation.  Instead, we use the fact that $\bfF(\S)$ is known to be a matroid.  (A side effect of this proof is that Theorem \ref{T:rank} becomes a proof that $\z$ is a module representation of $\M(\Ups)$ even when $\fG$ is not the additive group of a field.)  As before, we can treat the extra point $e_\infty$ as if it were a non-neutral half edge, so we are simultaneously proving that $\M_\infty(\Ups)$ is a matroid.

\begin{thm}\label{T:matroid}
The function $\rk_\Ups$ is a matroid rank function on $E(\Ups)$ and $E_\infty(\Ups)$.
\end{thm}

\begin{proof}
We establish the fundamental properties of a matroid rank function on $E$, which are:
\begin{enumerate}[\quad R1.]
\item \emph{Normalization}:  $\rk(\eset) = 0$.
\item \emph{Unit Increase}:  If $e \notin S \subset E$, then $\rk S \leq \rk(S \cup \{e\}) \leq \rk S + 1$.
\item \emph{Submodularity}:  If $e,f \notin S \subset E$, then 
$$\rk(S \cup \{e\}) - \rk S \geq \rk(S \cup \{e,f\}) - \rk(S \cup \{f\}).$$
\end{enumerate}

R1 is obvious.

For R2, what needs to be proved is that
$$
\rk_\S S + \delta_\Ups(S) \leq \rk_\S(S \cup \{e\}) + \delta_\Ups(S \cup \{e\}) \leq \rk_\S S + \delta_\Ups(S) + 1.
$$
The left-hand inequality is true because $\rk_\S$ and $\delta_\Ups$ are weakly increasing.  
The right-hand inequality is true by property R2 of $\rk_\S$ if $\delta_\Ups(S) = \delta_\Ups(S \cup \{e\})$, so assume $S$ is hyperbalanced and $S \cup \{e\}$ is hyperfrustrated.  Then $\rk_\S(S \cup \{e\}) = \rk_\S S$ by Lemma \ref{L:addedgebf}.  That proves R2.

\begin{lem}\label{L:addedgebf}
Suppose $S \subset E$ is hyperbalanced, $e \in E$, and $S \cup \{e\}$ is hyperfrustrated.  Then $e \in \clos_\S S$, $rk_\S(S \cup \{e\}) = \rk_\S S$, and $rk_\Ups(S \cup \{e\}) = \rk_\Ups S + 1$.
\end{lem}

\begin{proof}
There is a non-neutral sign circuit $C \subseteq S \cup \{e\}$ that contains $e$.  Since $C$ exists, $e \in \clos_\S S$, so $\rk_\S(S \cup \{e\}) = \rk_\S S$.  Then $\rk_\Ups(S \cup \{e\})$ follows from Definition \ref{D:rank}.
\end{proof}

To prove R3 we use a simpler equivalent form \cite[Theorem 1.4.14]{Oxley}:
\begin{enumerate}[\quad R1.]
\item[R3$'$.] If $e,f \notin S \subset E$ and $\rk(S \cup \{e\}) = \rk(S \cup \{f\}) = \rk S$, then $\rk(S \cup \{e,f\}) = \rk S.$
\end{enumerate}
Applying this to $\rk_\Ups$, if all the sets in this formula are hyperbalanced, or all are hyperfrustrated, then $\delta_\Ups$ drops out and we have a known property of $\rk_\S$.  If $S$ is hyperbalanced, by Theorem \ref{L:hyperbalneutral} we may assume every edge in $S$ is neutral.  Lemma \ref{L:addedgebf} implies that $e, f \in \clos_\S S$ and also that $S \cup \{e\}$ and $S \cup \{f\}$ are hyperbalanced.  Thus, there exist sign circuits $C_e$ and $C_f$ such that $e \in C_e \subseteq S \cup \{e\}$ and $f \in C_f \subseteq S \cup \{f\}$, both of which are neutral.  Therefore $e$ and $f$ are neutral edges, so $S \cup \{e,f\}$ is hyperbalanced and R3$'$ is proved.
\end{proof}

As we proved the matroid properties subsequent to Section \ref{sec:rank} without reference to the vector representation, they are all true in general with suitably adapted notation and the same proofs.

\sectionpage
\section*{Data statement}

There are no associated data.



\begin{thebibliography}{99}

\bibitem{Ard} Federico Ardila, 
Semimatroids and their Tutte polynomials.  
\emph{Rev.\ Colombiana Mat.}\ 41 (2007), no.\ 1, 39--66.
MR 2355665 (2008j:05082).  Zbl {1136.05008}.  

\bibitem{Athan} Christos A.\ Athanasiadis,
Characteristic polynomials of subspace arrangements and finite fields.  
\emph{Adv.\ Math.}\ {\bf 122} (1996), 193--233.  
MR 97k:52012.  Zbl 872.52006.

\bibitem{Bala}  A.R.\ Balasubramanian, 
Generalized threshold arrangements.  Submitted.  \emph{}\ 
arXiv:1904.08903.

\bibitem{ECR} Pascal Berthom\'e, Raul Cordovil, David Forge, V\'eronique Ventos, and Thomas Zaslavsky,  
An elementary chromatic reduction for gain graphs and special hyperplane arrangements.  
\emph{Electron.\ J.\ Combin.}\ {\bf 16} (1) (2009), article R121, 31 pp.
MR 2546324 (2010k:05253).  Zbl 1188.05076.

\bibitem{BZ} Ethan D.\ Bolker and Thomas Zaslavsky, 
A simple algorithm that proves half-integrality of bidirected network programming.  
\emph{Networks} {\bf 48} (2006), no.\ 1, 36--38.
MR 2007b:05098.  Zbl 1100.05046.

\bibitem{CDK} Tianran Chen, Robert Davis, and Evgeniia Korchevskaia, 
Facets and facet subgraphs of adjacency polytopes.  Submitted.
arXiv:2107.12315.

\bibitem{CM} Tianran Chen and Dhagash Mehta, 
On the network topology dependent solution count of the algebraic load flow equations. 
\emph{IEEE Trans.\ Power Syst.}\ 33 (2018), no.\ 2, 1451--1460.
arXiv:1512.04987.

\bibitem{Greene} C.\ Greene,
Acyclic orientations (Notes).
In: M.\ Aigner, ed., \emph{Higher Combinatorics} (Proc.\ NATO Adv.\ Study Inst., Berlin, 1976), pp.\ 65--68. 
NATO Adv.\ Study Inst.\ Ser., Ser. C, Vol.\ 31. 
D.\ Reidel, Dordrecht, 1977.
MR 58 \#27507 (book).  Zbl 389.05036.

\bibitem{Grunbaum} Branko Gr\"unbaum,
\emph{Convex Polytopes}.  
Interscience, New York, 1967.
MR 37 \#2085.  Zbl 163.16603.

\bibitem{MH} Tetsushi Matsui, Akihiro Higashitani, Yuuki Nagazawa, Hidefumi Ohsugi, and Takayuki Hibi,  
Roots of Ehrhart polynomials arising from graphs.  
\emph{J.\ Algebraic Combin.}\ 34 (2011), 721--749.
MR 2842918 (2012i:52024).  Zbl 1229.05122.

\bibitem{OH1} Hidefumi Ohsugi and Takayuki Hibi, 
Normal polytopes arising from finite graphs.  
\emph{J.\ Algebra} 207 (1998), 409--426.
MR 1644250 (2000a:13010).  Zbl 926.52017.

\bibitem{OT} Peter Orlik and Hiroaki Terao,  
\emph{Arrangements of Hyperplanes}.
Springer-Verlag, Berlin, 1992.
MR 94e:52014.  Zbl 757.55001.

\bibitem{Oxley} J.G.\ Oxley,
\emph{Matroid Theory}.  
Oxford University Press, Oxford, 1992.
MR 94d:05033.  Zbl 784.05002.
Second ed., 2011.
MR 2849819 (2012k:05002).  Zbl 1254.05002.

\bibitem{Seo1}  Seunghyun Seo, 
Shi threshold arrangement.  
\emph{Electron.\ J.\ Combin.}\ 19 (2012), no.\ 3, article P39, 9 pp.
MR {2988861}.  Zbl {1257.52009}.

\bibitem{Seo2}  ------, 
The Catalan threshold arrangement.  
\emph{J.\ Integer Seq.}\ 20 (2017), article 17.1.1, 12 pp.
MR {3606971}.  Zbl {1354.52026}.

\bibitem{Song1}  Joungmin Song, 
On certain hyperplane arrangements and colored graphs.  
\emph{Bull.\ Korean Math.\ Soc.}\ 54 (2017), no.\ 2, 375--382.
MR {3632442}.  Zbl {1373.32022}.

\bibitem{Song2}  ------, 
Enumeration of graphs and the characteristic polynomial of the hyperplane arrangements $\cJ_n$.  
\emph{J.\ Korean Math.\ Soc.}\ 54 (2017), no.\ 5, 1595--1604.
MR {3691940}.  Zbl {06853526}.

\bibitem{Song3}  ------, 
Characteristic polynomial of the hyperplane arrangements $\cJ_n$ via finite field method.  
\emph{Commun.\ Korean Math.\ Soc.}\ 33 (2018), no.\ 3, 759--765.
MR {3846025}.  Zbl {1401.32023}.

\bibitem{Song4}  ------, 
Characteristic polynomial of certain hyperplane arrangements through graph theory.  
arXiv:1701.07330.

\bibitem{WW} Michelle L.\ Wachs and James W.\ Walker, 
On geometric semilattices. 
\emph{Order} {\bf 2} (1986), 367--385.
MR 87f:06004.  Zbl 589.06005.

\bibitem{FUTA} Thomas Zaslavsky,
\emph{Facing Up to Arrangements: Face-Count Formulas for Partitions of Space by Hyperplanes}.
Mem.\ Amer.\ Math.\ Soc., No.\ 154 (= Vol.\ 1, Issue 1).  
American Mathematical Society, Providence, R.I., 1975.
MR {357135} (50 \#9603).  Zbl {296.50010}.

\bibitem{SG} ------, 
Signed graphs. 
\emph{Discrete Appl.\ Math.}\ {\bf 4} (1982), 47--74.  
Erratum. 
\emph{Discrete Appl.\ Math.}\ {\bf 5} (1983), 248.  
MR 84e:05095.  Zbl 503.05060.

\bibitem{SGC}  ------, 
Signed graph coloring. 
\emph{Discrete Math.}\ {\bf 39} (1982), 215--228.  
MR 84h:05050a.  Zbl 487.05027.

\bibitem{OSG}  ------, 
Orientation of signed graphs. 
\emph{Europ.\ J.\  Combin.}\ {\bf 12} (1991), 361--375.  
MR {1120422} (93a:05065).  Zbl {761.05095}.

\bibitem{BG1}  ------, 
Biased graphs.  I.\  Bias, balance, and gains.  
\emph{J.\ Combin.\ Theory Ser.\ B} {\bf 47} (1989), 32--52.  
MR 90k:05138.  Zbl 714.05057.

\bibitem{BG2} ------, 
Biased graphs.  II.\  The three matroids.  
\emph{J.\ Combin.\ Theory Ser.\ B} {\bf 51} (1991), 46--72.  
MR 91m:05056.  Zbl 763.05096.

\bibitem{BG3}  ------, 
Biased graphs.  III.\  Chromatic and dichromatic invariants.  
\emph{J.\ Combin.\ Theory Ser.\ B} {\bf 64} (1995), 17--88.
MR 96g:05139.  Zbl 857.05088.

\bibitem{BG4}  ------, 
Biased graphs IV:  Geometrical realizations.  
\emph{J.\ Combin.\ Theory Ser.\ B} {\bf 89} (2003), no.\ 2, 231--297.
MR 2005b:05057.  Zbl 1031.05034.

\bibitem{PDS} ------, 
Perpendicular dissections of space.  
\emph{Discrete Comput.\ Geom.}\ {\bf 27} (2002), 303--351.
MR 2003i:52026.  Zbl 1001.52011.

\end{thebibliography}
\end{document}